\documentclass[11pt]{amsart}
\usepackage{amsmath,amsfonts,amssymb,upgreek, amscd, mathtools}
\usepackage[breaklinks]{hyperref}

\theoremstyle{thrm}
\usepackage{tikz}
\usepackage{enumerate}
\usetikzlibrary{calc,decorations.markings}
\usepackage{tikz-cd}
\usepackage{graphicx}

\theoremstyle{plain}
\newtheorem{no}{Notation}[section]

\newtheorem{thm}{Theorem}[section]
\newtheorem{lemma}[thm]{Lemma}
\newtheorem{prop}[thm]{Proposition}

\newtheorem{cor}[thm]{Corollary}

\newtheorem{defn}[thm]{Definition}
\newtheorem{example}[thm]{Example}
\newtheorem*{ack}{Acknowledgement}

\theoremstyle{definition}
\newtheorem{remark}[equation]{Remark}

\newcommand{\B}{\operatorname{B} }
\newcommand{\C}{\operatorname{C} }

\newcommand{\Ho}{\operatorname{H} }

\newcommand{\Z}{\operatorname{Z} }

\newcommand{\im}{\operatorname{im} }

\newcommand{\Fix}{\operatorname{Fix}}

\newcommand{\Id}{\operatorname{id}}

\newcommand{\Aut}{\operatorname{Aut} }

\newcommand{\Ext}{\operatorname{Ext} }

\newcommand{\Map}{\operatorname{Map} }

\newcommand{\Ker}{\operatorname{ker} }

\newcommand{\IM}{\operatorname{im} }

\numberwithin{equation}{section}

\begin{document}
	
	\title{A Wells-like exact sequence for abelian extensions of relative Rota--Baxter groups}
	
	\author{Pragya Belwal}
	\address{Department of Mathematical Sciences, Indian Institute of Science Education and Research (IISER) Mohali, Sector 81, SAS Nagar, P O Manauli, Punjab 140306, India} 
	\email{pragyabelwal.math@gmail.com}
	
	\author{Nishant Rathee}
	\address{Department of Mathematical Sciences, Indian Institute of Science Education and Research (IISER) Mohali, Sector 81, SAS Nagar, P O Manauli, Punjab 140306, India} 
	\email{nishantrathee@iisermohali.ac.in, monurathee2@gmail.com}


	\begin{abstract}
Relative Rota--Baxter groups, a generalization of Rota--Baxter groups, are closely connected to skew left braces, which play a fundamental role in understanding non-degenerate set-theoretical  solutions to the Yang-Baxter equation. In this paper, we explore the inducibility problem for automorphisms of abelian extensions of relative Rota--Baxter groups. This problem is intricately linked to the recently introduced second cohomology of relative Rota--Baxter groups. Specifically, we prove a Wells-like exact sequence for abelian extensions of relative Rota--Baxter groups. The sequence establishes a connection among the group of derivations, certain automorphism group, and the second cohomology of relative Rota--Baxter groups, thereby giving precise structural relationships between these groups.
		
	\end{abstract}	
	
	\maketitle
	
	\section{Introduction}

Relative Rota--Baxter operators, also referred to as $\mathcal{O}$-operators, find their roots in the operator formulation of the classical Yang–Baxter equation, as introduced by Bai in \cite{CB13}. Jiang, Sheng, and Zhu recently extended the concept to Lie groups in \cite{JYC22}, presenting a generalization of Rota--Baxter operators with weight 1, initially defined in \cite{LHY2021}. The central theme of this development lies in the ability to further differentiate these operators, leading to the emergence of relative Rota--Baxter operators on corresponding Lie algebras. The relationship between relative Rota--Baxter groups and other algebraic structures, as well as their role in solving the set-theoretical Yang-Baxter equation, is explored in \cite{BGST}.

Bardakov and Gubarev explored Rota--Baxter operators on (abstract) groups in \cite{VV2022, VV2023}, demonstrating that a Rota--Baxter operator on a group gives rise to a skew left brace structure on that group. The investigation of Rota--Baxter groups within the context of skew left braces is further carried out in \cite{AS22}, where a non-trivial example of a skew left brace that cannot be derived from Rota--Baxter operators on groups has been given.  The ideas introduced in \cite{VV2022} are further employed in \cite{NM1} to observe that the category of relative Rota--Baxter groups serves as a generalization of the category of skew left braces, since the category of bijective relative Rota--Baxter groups is isomorphic to the category of skew left braces.

The Wells exact sequence, originally introduced by Wells in \cite{W71}, serves as a crucial tool for addressing the inducibility problem for automorphisms associated with group extensions. The sequence has undergone adaptation and generalisations for various algebraic structures, enabling the exploration of automorphisms of extensions specific to those categories. In \cite{BS2017, HH20}, the inducibility problem has been explored for Lie algebra and Lie superalgebra extensions, leading to the derivation of a Wells-like  exact sequence customized for these extensions. Additionally, a parallel examination of the inducibility problem for extensions of Lie-Yamaguti algebras is presented in \cite{GMM23}. Analogous sequences for set-theoretical solutions of  Yang-Baxter equation are also developed. For example, a Wells-like  exact sequence for skew braces is detailed in \cite{NMY1, N1}, and a Wells-like exact sequence for the cohomology theory of linear cycle sets, developed in \cite{LV16}, is defined in \cite{BS2023}. A similar  exact sequence for quandle extensions has been developed in \cite{BS2021}.   The inducibility and lifting problems concerning the extension of Rota--Baxter and relative Rota--Baxter operators across various algebraic structures have been investigated by numerous authors, see \cite{MDH23, AN2, AN1}. 

In this paper, we delve into the inducibility problem concerning automorphisms of abelian extensions of relative Rota--Baxter groups. The organization of the paper is as follows: in Section \ref{Preliminaries}, we present the preliminaries of relative Rota--Baxter groups, which are essential for this work. We revisit some foundational facts about extensions of relative Rota--Baxter groups in Section \ref{Extensions and second cohomology of relative Rota--Baxter groups}, focusing on the relationship between the second cohomology group and extensions of relative Rota--Baxter groups. In Section \ref{Wells-like exact sequence for relative Rota--Baxter groups}, we introduce some automorphism groups for extensions of relative Rota--Baxter groups, and define an action of these groups on extensions. We prove that the automorphism group of an extension of relative Rota--Baxter groups contains those automorphisms that induce identity on the substructure, and the quotient is isomorphic to the set of derivations (Theorem \ref{wells6 RRB}). This result leads to a Wells-like exact sequence for extensions of relative Rota--Baxter groups (Theorem \ref{wells7 RRB}). In the end, we provide a module-theoretic criterion for an automorphism of a relative Rota-Baxter group to be inducible (Theorem \ref{moduletheory}).

\medskip

\section{Preliminaries}\label{Preliminaries}

In this section, we revisit some basic concepts on relative Rota--Baxter groups essential for our study. For an in-depth understanding, we direct the reader to \cite{JYC22, NM1}. Our terminology aligns with that of \cite{BRS2023, NM1}, which is distinct from other sources.

\begin{defn}
	A relative Rota--Baxter group is a quadruple $(H, G, \phi, R)$, where $H$ and $G$ are groups, $\phi: G \rightarrow \Aut(H)$ a group homomorphism (where $\phi(g)$ is denoted by $\phi_g$) and $R: H \rightarrow G$ is a map satisfying the condition $$R(h_1) R(h_2)=R(h_1 \phi_{R(h_1)}(h_2))$$ for all $h_1, h_2 \in H$. 
	\par
	\noindent The map $R$ is referred to as the relative Rota--Baxter operator on $H$.
\end{defn}

We say that the relative Rota--Baxter group  $(H, G, \phi, R)$ is {\it trivial} if $\phi:G \to \Aut(H)$ is the trivial homomorphism. Further, it is said to be {\it bijective} if the Rota--Baxter operator $R$ is a bijection. We refer the reader to \cite{BRS2023, JYC22} for examples.

\begin{remark}
	Note that if $(H, G, \phi, R)$ is  a trivial relative Rota--Baxter group, then $R:H \rightarrow G$ is a group homomorphism. 
\end{remark}

Let  $(H, G, \phi, R)$ be a relative Rota--Baxter group, and let $K \leq H$ and $L \leq G$ be subgroups.
\begin{enumerate}
	\item If  $\phi_\ell(K) \subseteq K$ for all $\ell \in L$, then we denote the restriction of $\phi$ by $\phi|: L \to \Aut(K)$. 
	\item If $R(K) \subseteq L$, then we denote the restriction of $R$ by $R|: K \to L$.
\end{enumerate}

\begin{defn}
	Let $(H,G,\phi,R)$ be a relative Rota--Baxter group, and $K\leq H$ and $L\leq G$ be subgroups. Suppose that  $\phi_\ell(K) \subseteq K$ for all $\ell \in L$ and $R(K) \subseteq L$. Then $(K,L,\phi |,R |)$ is a relative Rota--Baxter group, which we refer as a relative Rota--Baxter subgroup of $(H,G,\phi,R)$ and write $(K,L,\phi |,R |)\leq(H,G,\phi,R)$.
\end{defn}

\begin{defn}\label{defn ideal rbb-datum}
	Let $(H, G, \phi, R)$ be a relative Rota--Baxter group and  $(K, L,  \phi|, R|) \leq (H, G, \phi, R)$ its relative Rota--Baxter subgroup. We say that $(K, L,  \phi|, R|)$ is an ideal of $(H, G, \phi, R)$ if 
	\begin{align}
		& K \trianglelefteq H \quad \mbox{and} \quad L \trianglelefteq G, \label{I0}\\
		& \phi_g(K) \subseteq K  \mbox{ for all } g \in G, \label{I1} \\
		& \phi_\ell(h) h^{-1} \in K \mbox{ for all } h \in H \mbox{ and }  \ell \in L. \label{I2}
	\end{align}
	We write $(K, L, \phi|, R|) \trianglelefteq (H, G, \phi, R)$ to denote an ideal of a relative Rota--Baxter group. 
\end{defn}

The preceding definitions lead to the following result \cite[Theorem 5.3]{NM1}.

\begin{thm}\label{subs}
	Let $(H, G, \phi, R)$ be a relative Rota--Baxter group and $(K, L,  \phi|, R|)$ an ideal of $(H, G, \phi, R)$. Then there are maps $\overline{\phi}: G/L \to \Aut(H/K)$ and  $\overline{R}: H/K \to G/L$ defined by
	$$ \overline{\phi}_{\overline{g}}(\overline{h})=\overline{\phi_{g}(h)} \quad \textrm{and} \quad	\overline{R}(\overline{h})=\overline{R(h)}$$
	for $\overline{g} \in G/L$ and $\overline{h} \in H/K$, such that  $(H/K, G/L, \overline{\phi}, \overline{R})$ is a relative Rota--Baxter group.
\end{thm}
\begin{no}
	We write $(H, G, \phi, R)/(K, L, \phi|, R|)$ to denote the quotient relative Rota--Baxter group $(H/K, G/L, \overline{\phi}, \overline{R})$.
\end{no}

\begin{defn}
	Let $(H, G, \phi, R)$ and $(K, L, \varphi, S)$ be two relative Rota--Baxter groups.
	\begin{enumerate}
		\item A homomorphism $(\psi, \eta): (H, G, \phi, R) \to (K, L, \varphi, S)$ of relative Rota--Baxter groups is a pair $(\psi, \eta)$, where $\psi: H \rightarrow K$ and $\eta: G \rightarrow L$ are group homomorphisms such that
		\begin{equation}\label{rbb datum morphism}
			\eta \; R = S \; \psi \quad \textrm{and} \quad \psi \; \phi_g =  \varphi_{\eta(g)}\psi
		\end{equation}
		for all $g \in G$.
		\item The kernel of a homomorphism $(\psi, \eta): (H, G, \phi, R) \to (K, L, \varphi, S)$ of relative Rota--Baxter groups is the quadruple $$(\Ker(\psi), \Ker(\eta), \phi|, R|),$$ where $\Ker(\psi)$ and $\Ker(\eta)$ denote the kernels of the group homomorphisms $\psi$ and $\eta$, respectively. The conditions in \eqref{rbb datum morphism} imply that the kernel is itself a relative Rota--Baxter group. In fact, the kernel turns out to be an ideal of $(H, G, \phi, R)$.
		
		\item The image of a homomorphism $(\psi, \eta):  (H, G, \phi, R) \to (K, L, \varphi, S)$ of relative Rota--Baxter groups is the quadruple 
		$$(\IM(\psi), \IM(\eta), \varphi|, S| ),$$ where $\IM(\psi)$ and $\IM(\eta)$ denote the images of the group homomorphisms $\psi$ and $\eta$, respectively. The image is itself a relative Rota--Baxter group.
		
		\item A homomorphism $(\psi, \eta)$ of relative Rota--Baxter groups is called an isomorphism if both $\psi$ and $\eta$ are group isomorphisms. Similarly, we say that $(\psi, \eta)$ is an embedding of a relative Rota--Baxter group if both $\psi$ and $\eta$ are embeddings of groups.
	\end{enumerate}
\end{defn}

\begin{defn}
	The center of a relative Rota--Baxter group $(H, G, \phi, R)$ is defined as
	$$\Z(H, G, \phi, R)= \big( \Z^{\phi}_R(H), \Ker(\phi), \phi|, R| \big),$$
	where $\Z^{\phi}_{R}(H)=\Z(H)  \cap \Ker(\phi \,R) \cap \Fix(\phi)$, $\Fix(\phi)= \{ x \in H  \mid \phi_g(x)=x \mbox{ for all } g \in G \}$ and  $R: H^{(\circ_R)} \rightarrow G$ is viewed as a group homomorphism.
\end{defn}
\medskip

\section{ Extensions and second cohomology of relative Rota--Baxter groups}\label{Extensions and second cohomology of relative Rota--Baxter groups}

\begin{defn}
Let $(K,L, \alpha,S )$ and $(A,B, \beta, T)$ be relative Rota--Baxter groups.  
\begin{enumerate}
	\item An extension of $(A,B, \beta, T)$ by $(K,L, \alpha,S )$ is a relative Rota--Baxter group  $(H,G, \phi, R)$ that fits into the sequence 
	$$\mathcal{E} : \quad  {\bf 1} \longrightarrow (K,L, \alpha,S ) \stackrel{(i_1, i_2)}{\longrightarrow}  (H,G, \phi, R) \stackrel{(\pi_1, \pi_2)}{\longrightarrow} (A,B, \beta, T) \longrightarrow {\bf 1},$$ 
	where  $(i_1, i_2)$ and $(\pi_1, \pi_2)$ are morphisms of relative Rota--Baxter groups such that $(i_1, i_2)$ is an embedding, $(\pi_1, \pi_2)$ is an epimorphism of  relative Rota--Baxter groups and $(\IM(i_1), \IM(i_2), \phi|, R|)= (\Ker(\pi_1), \Ker(\pi_2), \phi|, R|)$.
	\item[] To avoid complexity of notation, we assume that $i_1$ and $i_2$ are inclusion maps. This allows us to write that $\phi$ restricted to  $L$ is $\alpha$ and $R$ restricted to  $K$ is $S$.	
	\item We say that $\mathcal{E}$ is an abelian extension if $K$ and $L$ are abelian groups and the relative Rota--Baxter group $(K,L, \alpha,S )$ is trivial.
\end{enumerate}
\end{defn}
Here, ${\bf 1}$ denotes the trivial relative Rota--Baxter group for which both the underlying groups are trivial.

\begin{example}
	Let $\mathcal{K}=(K,L,\alpha,S) $ and $\mathcal{A}=(A,B,\beta,T)$ be two relative Rota--Baxter groups. Then the direct product of $\mathcal{K}$ and $\mathcal{A}$ is defined as $$\mathcal{A}\times\mathcal{K}=(A\times K, \, B\times L, \, \beta\times \alpha, \, T\times S),$$ where $A\times K$ and $B\times L$ are direct product of groups, and the action and the relative  Rota--Baxter operator are given by
	\begin{eqnarray*}
		(\beta\times\alpha)_{(b,l)}(a,k)&=&(\beta_b(a),\,\alpha_l(k)),\\
		(T\times S)(a,k)&=&(T(a),\,S(k)),
	\end{eqnarray*}
	for $a \in A$, $b \in B$, $k \in K$ and $l \in L$.
\end{example}

\begin{prop}\cite[Proposition 3.6]{BRS2023}
	An extension
	$$\mathcal{E} : \quad {\bf 1} \longrightarrow (K,L, \alpha,S ) \stackrel{(i_1, i_2)}{\longrightarrow}  (H,G, \phi, R) \stackrel{(\pi_1, \pi_2)}{\longrightarrow} (A,B, \beta, T) \longrightarrow {\bf 1}$$
	of relative Rota--Baxter groups	induces extensions of groups $$\mathcal{E}_1 : \quad  1 \longrightarrow K \stackrel{i_1}{\longrightarrow}  H \stackrel{\pi_1 }{\longrightarrow} A \longrightarrow 1 \quad \textrm{and} \quad \mathcal{E}_2 : \quad 1 \longrightarrow L \stackrel{i_2}{\longrightarrow}  G \stackrel{\pi_2 }{\longrightarrow} B \longrightarrow 1.$$
	Furthermore, $ (K,L, \alpha,S )$ is an ideal of $(H,G, \phi, R)$ and the quotient relative Rota--Baxter group $(H, G, \phi, R)/ (K,L, \alpha,S ) $ is isomorphic to $(A, B, \beta, T).$
\end{prop}

\begin{remark}
Let $(H, G, \phi, R)$ be a relative Rota--Baxter group, and $(K, L, \phi|, R|)$ be its ideal. The relative Rota--Baxter group $(H, G, \phi, R)$ is an extension of $(K, L, \phi|, R|)$ by the quotient relative Rota--Baxter group $(H, G, \phi, R) / (K, L, \phi|, R|)$. In particular, every relative Rota--Baxter group is an extension of its center.
\end{remark}

Next, we recall some necessary results on abelian extensions of relative Rota--Baxter groups \cite[p.8]{BRS2023}. Throughout our discussion, we denote by $\mathcal{E}$ an abelian extension defined as follows
$$ {\bf 1} \longrightarrow (K,L, \alpha,S ) \stackrel{(i_1, i_2)}{\longrightarrow}  (H,G, \phi, R) \stackrel{(\pi_1, \pi_2)}{\longrightarrow} (A,B, \beta, T) \longrightarrow {\bf 1},$$ 
and $(s_H, s_G)$ denote the  set-theoretic section to the extension $\mathcal{E}$. 
\begin{prop}\label{f rho chi eqn}\cite[p.11]{BRS2023}
	Let  $a \in A$, $b \in B$, $k \in K$, and $l \in L$. Then, the following statements hold:
	
	\begin{enumerate}
		\item The action $\phi$ is characterized by the equation
		\begin{equation}\label{feqn}
			\phi_{s_G(b)l}(s_H(a)k) = s_H(\beta{_b(a)}) \, \rho(a,b) \, \phi_{s_G(b)}(f(l,a)k),
		\end{equation}
		where $f: L \times A \rightarrow K$ is defined as $f(l,a) := s_H(a)^{-1} \phi_l(s_H(a))$, and $\rho: A \times B \rightarrow K$ is given by $\rho(a,b) := (s_H(\beta_b(a)))^{-1}\phi_{s_G(b)}(s_H(a))$.
		
		\item The relative Rota--Baxter operator $R$ is expressed as
		\begin{equation}\label{relativecon}
			R(s_H(a)k) = s_G(T(a)) \, \chi(a) \, S(\phi^{-1}_{s_G(T(a))}(k)),
		\end{equation}
		where $\chi: A \rightarrow K$ is defined by,  $\chi(a):=s_G(T(a))^{-1}R(s_H(a))$.
	\end{enumerate}
\end{prop}

\begin{lemma}\label{properties of f}\cite[Lemma 3.9]{BRS2023}
	Let $\mathcal{E}$ be an abelian extension of relative Rota--Baxter groups. Then the following hold:
	\begin{enumerate}
		\item The map $f: L \times A \longrightarrow K$ defined by
		$$ f(l,a)= s_H( a)^{-1} \phi_l(s_H(a))$$
		for $l \in L$ and $a \in A$, is independent of the choice of the section $s_H$.
		\item $f(l_1l_2,a) = f( l_1, a)f( l_2,a)$ for all $l_1, l_2 \in L$ and $a \in A$.
		\item $f(l, a_1 a_2) = \mu_{ a_2}(f(l,a_1))f(l, a_2)$ for all $l \in L$ and $a_1, a_2 \in A$.
	\end{enumerate}
\end{lemma}

\begin{prop}\label{construction of actions}\cite[Proposition 3.7]{BRS2023}
	Consider the abelian extension 
	$$\mathcal{E} : \quad  {\bf 1} \longrightarrow (K,L, \alpha,S ) \stackrel{(i_1, i_2)}{\longrightarrow}  (H,G, \phi, R) \stackrel{(\pi_1, \pi_2)}{\longrightarrow} (A,B, \beta, T) \longrightarrow {\bf 1}$$
	of relative Rota--Baxter groups. Then the following hold:
	\begin{enumerate}
		\item  The map  $\nu: B \rightarrow \Aut(K)$ defined by 
		\begin{equation}\label{nuact}
			\nu_b(k):= \phi_{s_G(b)}(k)
		\end{equation}
		for $b \in B$ and $k \in K$, is a homomorphism of groups.
		\item  The $\mu: A \rightarrow \Aut(K)$ defined by 
		\begin{equation}\label{muact}
			\mu_a(k):=s_H(a)^{-1}\, k\, s_H(a)
		\end{equation}
		for $a \in A$ and $ k \in K$, is an anti-homomorphism of groups.
		\item The map $\sigma: B \rightarrow \Aut(L)$ defined by 
		\begin{equation}\label{sigmaact}
			\sigma_b(l):=s_G(b)^{-1}\,l\, s_G(b)
		\end{equation}
		for $b \in B$ and $ l \in K$, is an anti-homomorphism of groups.
		
		\item[] 	Further, all the maps are independent of the choice of a section to $\mathcal{E}$.
			\item The map $\tau_1: A \times A \rightarrow K$ given by
		\begin{equation}\label{mucocycle}
			\tau_1(a_1, a_2):= s_H(a_1 a_2)^{-1}s_H(a_1)s_H(a_2)
		\end{equation}
		for $a_1, a_2 \in A$ is a group 2-cocycle with respect to the action $\mu$.
		\item The map $\tau_2: B \times B \rightarrow L$ given by
		\begin{equation}\label{sigmacocycle}
			\tau_2(b_1, b_2):= s_G(b_1 b_2)^{-1}s_G(b_1)s_G(b_2)
		\end{equation}
	for	$b_1, b_2 \in B$ is a group 2-cocycle with respect to the action $\sigma$.
	\end{enumerate}
\end{prop}
By examining the relationships established among $\nu, \mu, \sigma,  f$, and their corresponding properties outlined in Lemma \ref{properties of f} and Proposition \ref{construction of actions},  the following definition of a module over a relative Rota--Baxter group has been introduced in \cite{BRS2023}.
\begin{defn}\label{moddefn}
	A module over a relative Rota--Baxter group $(A, B,\beta, T)$ is a trivial relative Rota--Baxter group $(K, L, \alpha, S)$ such that there exists a quadruple $(\nu, \mu, \sigma, f)$ (called action) of maps satisfying the following conditions:
	\begin{enumerate}
		\item The group $K$ is a left $B$-module and a right $A$-module with respect to the actions $\nu:B \to \Aut(K) $ and $\mu: A \to \Aut(K)$, respectively.
		\item The group $L$ is a right $B$-module with respect to the action $\sigma: B \to \Aut(L)$.
		\item The  map  $f: L \times A \to K$ has the property that $f(-,a): L \rightarrow K$ is a homomorphism for all $a \in A$ and $f(l,-): A \rightarrow K$ is a derivation with respect to the action $\mu$ for all $l \in L$.
		\item $S\big(\nu^{-1}_{T(a)}(\mu_{a}(k)) \,\nu^{-1}_{T( a)}(f(S(k), a))\big)=\;\sigma_{T(a)}(S(k))$ for all $ a \in A$ and $k \in K$.
		\item $\nu_b(\mu_{a}(k))= \;   \mu_{ \beta{_{b}(a)}}( \nu_{b} (k))$ for all $a \in A$, $b \in B$ and $k \in K$. 
	\end{enumerate}
\end{defn}

Let $G$ be any group  and $M$ an abelian group , let $C^n(G, M)$ denote the group of maps from $G^n$ to $M$ that vanish on degenerate tuples. Similarly for  any group $H$, let $C(G \times H, M)$ be the group of maps from $G \times H$ to $M$ that vanish on degenerate tuples. In other words, for any $f \in C(G \times H, M)$, it satisfies the conditions $f(g, 1_H) = f(1_G, h) = 1_M$, where $1_K$ represents the identity element of a group $K$.

 Consider a module $\mathcal{K} = (K, L, \alpha, S)$ over $\mathcal{A} = (A, B, \beta, T)$ with  action  $(\nu, \mu, \sigma, f)$. We set the following groups
\begin{eqnarray*}
\C^{1}_{RRB}(\mathcal{A}, \mathcal{K}) &:=& C^1(A, K) \oplus C^1(B,L),\\
\C^{2}_{RRB}(\mathcal{A}, \mathcal{K}) &:=& C^2(A, K) \oplus C^2(B,L) \oplus C(A \times B, K) \oplus C(A,L).
\end{eqnarray*}

Consider the subgroup $\Z^1_{RRB}(\mathcal{A}, \mathcal{K})$ of $\C^{1}_{RRB}(\mathcal{A}, \mathcal{K})$, comprising pairs $(\kappa_1, \kappa_2) \in \C^{1}_{RRB}$ that satisfy the following conditions:
\begin{eqnarray}
 \kappa_1(a_1 a_2)&=&\kappa_1(a_2) \mu_{a_2} (\kappa_1(a_1)), \label{der1}\\
\kappa_2(b_1 b_2) &= & \kappa_2(b_2) \sigma_{b_2} (\kappa_2(b_1)), \label{der2}\\
\nu_b \big(f(\kappa_2(b_1), a_1)\kappa_1(a_1) \big) & =  & \kappa_1(\beta_{b_1}(a_1)), \label{der3}\\
S \big(\nu^{-1}_{T(a_1)}(\kappa_1(a_1)) \big) & = & \kappa_2(T(a_1)) \label{der4} 
\end{eqnarray}
for all $a_1, a_2 \in A$ and $b_1, b_2 \in B.$

Let $\Z^2_{RRB}(\mathcal{A}, \mathcal{K})$ be a subgroup of $\C^{2}_{RRB}(\mathcal{A}, \mathcal{K})$ consisting of elements $(\tau_1, \tau_2, \rho, \chi) \in \C^{2}_{RRB}(\mathcal{A}, \mathcal{K})$ that satisfy the following conditions for all $a_1, a_2, a_3 \in A$ and $b_1, b_2, b_3 \in B.$
\begin{eqnarray}
\tau_1(a_2, a_3) \tau_1( a_1, a_2 a_3) &=& \tau_1(a_1 a_2, a_3) \mu_{a_3}(\tau_1(a_1, a_2)),\label{cocycle1}\\
\tau_2(b_2, b_3) \tau_2( b_1, b_2 b_3) &=& \tau_2(b_1 b_2, b_3) \sigma_{b_3}(\tau_2(b_1, b_2)),\label{cocycle2}\\
 \rho(\beta_{b_2}(a_1), b_1)  \,\nu_{b_1}(\rho(a_1,b_2)) &=& \rho(a_1, b_1 b_2) \, \nu_{b_1 b_2}( f(\tau_2(b_1, b_2), a_1)) ,\label{cocycle3}\\
\rho(a_1 a_2, b_1) \, \nu_{b_1}(\tau_1(a_1, a_2))   &=& \mu_{ \beta_{b_1}(a_2)}(\rho(a_1,b_1))\; \rho(a_2, b_1) \;  \tau_1(\beta_{b_1}(a_1), \beta_{b_1}(a_2))  ,\label{cocycle4}\\
 \tau_2(T(a_1), T(a_2) ) \delta^1_{\sigma}(\chi)(a_1, a_2) &= &  S \big(\nu^{-1}_{T(a_1 \circ_T a_2)}\big(\rho(a_2, T(a_1))  \, \tau_1(a_1,\beta_{T(a_1)}(a_2)) \nonumber \\
 && \,\nu_{T(a_1)}(f(\chi(a_1), a_2))  \big)\big)\label{cocycle5},
\end{eqnarray}
where $\delta^1_{\sigma}(\chi)(a_1, a_2)= \chi(a_2) \chi(a_1 \circ_T a_2)^{-1} \sigma_{T(a_2)}(\chi(a_1)). $ 

Next, consider the subgroup $\B^2_{RRB}(\mathcal{A}, \mathcal{K})$ of $\Z^2_{RRB}(\mathcal{A}, \mathcal{K})$, consisting of elements $(\tau_1, \tau_2, \rho, \chi)$ such that there exist $\kappa_1, \kappa_2 \in  \C^{1}_{RRB}(\mathcal{A}, \mathcal{K})$ and satisfy the following
\begin{eqnarray*}
\tau_1(a_1, a_2) &=&\kappa_1(a_1 a_2)^{-1}\kappa_1(a_2) \mu_{a_2} (\kappa_1(a_1)),\\
\tau_2(b_1, b_2) &=& \kappa_2(b_1 b_2)^{-1} \kappa_2(b_2) \sigma_{b_2} (\kappa_2(b_1)),\\
\rho(a_1, a_2) &=&\nu_b\big( f(\kappa_2(b_1), a_1)\kappa_1(a_1) \big) \big(\kappa_1(\beta_{b_1}(a_1))\big)^{-1},\\
 \chi(a_1) &=& S \big(\nu^{-1}_{T(a_1)}(\kappa_1(a_1)) \big) \big( \kappa_2(T(a_1)) \big)^{-1}
\end{eqnarray*}
for all $a_1, a_2 \in A$ and $b_1, b_2 \in B.$

We define the second cohomology group of $\mathcal{A}=(A, B, \beta, T)$  with coefficients in $\mathcal{K}=(K, L, \alpha, S)$ by  $$\Ho^2_{RRB}(\mathcal{A}, \mathcal{K}):=\Z^2_{RRB}(\mathcal{A}, \mathcal{K})/  \B^2_{RRB}(\mathcal{A}, \mathcal{K}).$$

 Let $\Ext(\mathcal{A},\mathcal{K})$ denote the set of all the equivalence classes of extensions of $\mathcal{A}$ by $\mathcal{K}$ and let $\Ext_{(\nu, \mu, \sigma, f)}(\mathcal{A}, \mathcal{K})$ denote the set of equivalence classes of extensions of $\mathcal{A}$ by $\mathcal{K}$ for which the associated action is $(\nu, \mu, \sigma, f)$. Then we can write
 $$\Ext(\mathcal{A}, \mathcal{K})=\bigsqcup_{(\nu, \mu, \sigma, f)} \Ext_{(\nu, \mu, \sigma, f)}(\mathcal{A}, \mathcal{K}).$$
  We have the following result.

\begin{thm}\label{ext and cohom bijection}\cite[Theorem 3.18]{BRS2023}
	Let $\mathcal{A}= (A,B, \beta, T)$ be a relative Rota--Baxter group and $\mathcal{K}=  (K,L,\alpha,S )$ a trivial relative Rota--Baxter group, where $K$ and $L$ are abelian groups.  Let  $(\nu, \mu, \sigma, f)$ be the quadruple of actions that makes $\mathcal{K}$ into an $\mathcal{A}$-module. Then there is a bijection between $\Ext_{(\nu, \mu, \sigma, f)}(\mathcal{A}, \mathcal{K})$ and $\Ho^2_{RRB}(\mathcal{A}, \mathcal{K})$.  
\end{thm}
\medskip

	\section{Wells-like exact sequence for relative Rota--Baxter groups}\label{Wells-like exact sequence for relative Rota--Baxter groups}
In this section, we establish a Wells-like exact sequence for abelian extensions of relative Rota--Baxter groups. Let 
$$\mathcal{E} : \quad  {\bf 1} \longrightarrow \mathcal{K}=(K,L, \alpha,S ) \stackrel{(i_1, i_2)}{\longrightarrow}  \mathcal{H}=(H,G, \phi, R) \stackrel{(\pi_1, \pi_2)}{\longrightarrow}\mathcal{A}= (A,B, \beta, T) \longrightarrow {\bf 1}$$ be an abelian extension of $\mathcal{A}$ by $\mathcal{K}$. We first define a group action on the set  $\Ext(\mathcal{A}, \mathcal{K})$.
	
	Let $\Aut(\mathcal{A})$  and $\Aut(\mathcal{K})$ denote the group of automorphisms of relative Rota--Baxter groups $\mathcal{A}$ and  $\mathcal{K}$, respectively. We define an  action of  $\Aut(\mathcal{A}) \times \Aut(\mathcal{K})$ on $\Ext(\mathcal{A},\mathcal{K})$ as follows. For a pair of automorphisms $(\psi, \theta) \in \Aut(\mathcal{A}) \times \Aut(\mathcal{K})$, where $\psi=(\psi_1,\psi_2)$ and  $\theta=(\theta_1,\theta_2)$, we  define a new extension
$$\mathcal{E}^{(\psi, \theta)} : \quad  {\bf 1} \longrightarrow \mathcal{K} \stackrel{(\theta_1, \theta_2)}{\longrightarrow}  \mathcal{H'}=(H',G', \phi', R') \stackrel{(\psi_1^{-1}\pi_1, \psi_2^{-1}\pi_2)}{\longrightarrow}\mathcal{A}\longrightarrow {\bf 1}.$$

Consider two equivalent extensions $$\mathcal{E}_1 : \quad  {\bf 1} \longrightarrow \mathcal{K} \stackrel{(i_1, i_2)}{\longrightarrow}  \mathcal{H}_1=(H_1,G_1, \phi_1, R_1) \stackrel{(\pi_1, \pi_2)}{\longrightarrow}\mathcal{A} \longrightarrow {\bf 1},$$ $$\mathcal{E}_2 : \quad  {\bf 1} \longrightarrow \mathcal{K} \stackrel{(i'_1, i'_2)}{\longrightarrow}  \mathcal{H}_2=(H_2,G_2, \phi_2, R_2) \stackrel{(\pi'_1, \pi'_2)}{\longrightarrow}\mathcal{A} \longrightarrow {\bf 1}$$
of $\mathcal{A}$ by $\mathcal{K}$. Then, for any $(\psi, \theta) \in \Aut(\mathcal{A}) \times \Aut(\mathcal{K})$, we observe that the extensions $\mathcal{E}_1^{(\psi, \theta)}$ and $\mathcal{E}_2^{(\psi, \theta)}$ are also equivalent. Thus, we can define a map $\Ext(\mathcal{A},\mathcal{K}) \to \Ext(\mathcal{A},\mathcal{K})$ by 
	\begin{equation}\label{act1 RRB}
		[\mathcal{E}] \mapsto  [ \mathcal{E}^{(\psi, \theta)}].
	\end{equation}
	If $\psi$ and $\theta$ are identity automorphisms, then obviously $\mathcal{E}^{(\psi, \theta)} = \mathcal{E}$. It is also easy to see that  
	$$[\mathcal{E}]  ^{(\psi, \theta) (\psi^{\prime}, \theta^{\prime})}=  \big([\mathcal{E}]^{(\psi, \theta)}\big)^{ (\psi^{\prime}, \theta^{\prime})}.$$
Hence, the association \eqref{act1 RRB} gives a right action of the group $\Aut(\mathcal{A}) \times \Aut(\mathcal{K})$  on the set $\Ext(\mathcal{A},\mathcal{K})$.
	
We know that $$\Ext(\mathcal{A}, \mathcal{K})=\bigsqcup_{(\nu, \mu, \sigma, f)} \Ext_{(\nu, \mu, \sigma, f)}(\mathcal{A}, \mathcal{K}).$$
For the rest of this section, $\mathcal{K}$ is an $\mathcal{A}$-module via the action  $(\nu,\mu,\sigma,f)$. Define the set
	\begin{eqnarray}\label{act stb}
	\C_{(\nu,\mu,\sigma,f)} &=& \big\{ (\psi, \theta) \in \Aut(\mathcal{A}) \times \Aut(\mathcal{K}) \mid \nu_b=\theta_1^{-1} \nu_{\psi_{2} (b)} \theta_1 ,~ \mu_a=\theta_1^{-1} \mu_{\psi_{1} (a)} \theta_1 ,\\
	&&
	 \sigma_b=\theta_2^{-1} \sigma_{\psi_{2} (b)} \theta_2\textit{ and }\theta_1(f(l,a))=f(\theta_2(l),\psi_1(a)) \big\}\nonumber.
	 \end{eqnarray}
Then  $\C_{(\nu,\mu,\sigma,f)}$ is a subgroup of $\Aut(\mathcal{A}) \times \Aut(\mathcal{K})$ and it  acts on  $\Ext_{(\nu,\mu,\sigma,f)}(\mathcal{A}, \mathcal{K})$ by the same rule as given in \eqref{act1 RRB}. Also, the set $\Ext_{(\nu,\mu,\sigma,f)}(\mathcal{A}, \mathcal{K})$ is invariant under the action of  $\C_{(\nu,\mu,\sigma,f)}$.
	
	Next we consider an action of $ \C_{(\nu,\mu,\sigma,f)}$ on $\Ho^2_{RRB}(\mathcal{A},\mathcal{K})$.
	Let $(\psi, \theta) \in \Aut(\mathcal{A}) \times \Aut(\mathcal{K})$. Let $\tau_1 \in \Map(A^n, K)$, $\tau_2 \in \Map(B^n, L)$, where $n \ge 1$ is an integer. Define $\tau_1^{(\psi, \theta)} : A^n \to K$ and $\tau_2^{(\psi, \theta)} : B^n \to L$ by
\begin{eqnarray*}
\tau_1^{(\psi, \theta)}(a_1, a_2, \ldots, a_n) &=& \theta_1^{-1}\big(\tau_1(\psi_1(a_1), \psi_1(a_2), \ldots, \psi_1(a_n))\big),\\
\tau_2^{(\psi, \theta)}(b_1, b_2, \ldots, b_n) &=& \theta_2^{-1}\big(\tau_2(\psi_2(b_1), \psi_2(b_2), \ldots, \psi_2(b_n))\big).
\end{eqnarray*}
 
Also, define $\rho^{(\psi, \theta)}:A\times B\to K$ and $\chi^{(\psi, \theta)}:A\to L$ by 
\begin{eqnarray*}
\rho^{(\psi, \theta)}(a,b)&=&\theta_1^{-1}(\rho(\psi_1(a),\psi_2(b))),\\
\chi^{(\psi, \theta)}(a)&=&\theta_2^{-1}(\chi(\psi_1(a))).
\end{eqnarray*}

	It is not difficult to see that the group $\Aut(\mathcal{A}) \times \Aut(\mathcal{K})$ acts on the group $C^{n}(\mathcal{A},\mathcal{K})$ by automorphisms, given by the association
	\begin{eqnarray}
		\tau_1 \mapsto \tau_1^{(\psi, \theta)},~ 	\tau_2 \mapsto \tau_2^{(\psi, \theta)},\label{act2 RRB}\\
		\rho\mapsto\rho^{(\psi, \theta)},~ \chi\mapsto\chi^{(\psi, \theta)}.\label{act2 RRB1}
		\end{eqnarray}
 We are interested in the action of $\C_{ (\nu,\mu,\sigma,f)}$ on $\Ho_{RRB}^2(\mathcal{A}, \mathcal{K})$. The association \eqref{act2 RRB},\eqref{act2 RRB1} induces an action of $\C_{(\nu,\mu,\sigma,f)}$ on $	C^{2}_{RRB} = C^2(A, K) \oplus C^2(B,L) \oplus C(A \times B, K) \oplus C(A,L)$ by setting 
	\begin{equation}\label{act3 RRB}
		(\tau_1,\tau_2,\rho,\chi) \mapsto \big(\tau_1^{(\psi, \theta)}, \tau_2^{(\psi, \theta)},\rho^{(\psi, \theta)},\chi^{(\psi, \theta)}\big).
	\end{equation}

\begin{lemma}\label{lemma-act3 RRB}
The following assertions hold for each $(\psi, \theta) \in \C_{(\nu,\mu,\sigma,f)}$:
\begin{enumerate}
\item If $(\tau_1,\tau_2,\rho,\chi) \in \Z^2_{RRB}(\mathcal{A},\mathcal{K})$, then $\big(\tau_1^{(\psi, \theta)}, \tau_2^{(\psi, \theta)},\rho^{(\psi, \theta)},\chi^{(\psi, \theta)}\big) \in \Z^2_{RRB}(\mathcal{A},\mathcal{K})$.
\item
 If $(\tau_1,\tau_2,\rho,\chi)\in \B^2_{RRB}(\mathcal{A},\mathcal{K})$, then $\big(\tau_1^{(\psi, \theta)}, \tau_2^{(\psi, \theta)},\rho^{(\psi, \theta)},\chi^{(\psi, \theta)}\big) \in \B^2_{RRB}(\mathcal{A},\mathcal{K})$.
\end{enumerate}
	Hence, the association \eqref{act3 RRB} gives an action of $\C_{(\nu,\mu,\sigma,f)}$ on $\Ho_{RRB}^2(\mathcal{A}, \mathcal{K})$ by automorphisms, defined by
	$$[	(\tau_1,\tau_2,\rho,\chi)]^{(\psi, \kappa)} =[\big(\tau_1^{(\psi, \theta)}, \tau_2^{(\psi, \theta)},\rho^{(\psi, \theta)},\chi^{(\psi, \theta)}\big)].$$
\end{lemma}
\begin{proof}
If $(\tau_1,\tau_2,\rho,\chi) \in \Z^2_{RRB}(\mathcal{A},\mathcal{K})$, then, for all $a_1, a_2, a_3 \in A$, we have
	\begin{eqnarray}
		\tau_1^{(\psi,\theta)}(a_2, a_3) \tau_1^{(\psi,\theta)}( a_1, a_2 a_3)&=&	\theta_1^{-1}\big(\tau_1(\psi_1(a_2), \psi_1(a_3))\big)\theta_1^{-1}\big( \tau_1( \psi_1(a_1), \psi_1(a_2 a_3))\big)\nonumber\\
		&=&\theta_1^{-1}\big(\tau_1(\psi_1(a_2), \psi_1(a_3))\tau_1(\psi_1( a_1), \psi_1(a_2 )\psi_1(a_3))\big),\nonumber\\
		&&\text{since }\theta_1\in \Aut(K)\nonumber\\
		&=&\theta_1^{-1}\big(\tau_1(\psi_1(a_1) \psi_1(a_2),\psi_1( a_3)) \mu_{\psi_1(a_3)}(\tau_1(\psi_1(a_1), \psi_1(a_2)))\big),\nonumber\\
		&&	\text{using }\eqref{cocycle1}\nonumber\\
		&=&\theta_1^{-1}\big(\tau_1(\psi_1(a_1 a_2),\psi_1( a_3))\big)\theta_1^{-1}\big(\mu_{\psi_1(a_3)}(\tau_1(\psi_1(a_1), \psi_1(a_2)))\big)\nonumber\\
		&=&\theta_1^{-1}\big(\tau_1(\psi_1(a_1 a_2),\psi_1( a_3))\big)\mu_{a_3}\big(\theta_1^{-1}(\tau_1(\psi_1(a_1), \psi_1(a_2)))\big),\nonumber\\
		&&\text{using \eqref{act stb} }\nonumber\\
		&=&\tau_1^{(\psi,\theta)}(a_1 a_2, a_3) \mu_{a_3}(\tau_1^{(\psi,\theta)}(a_1, a_2)).\label{z2 condn1}
	\end{eqnarray}
Similarly, we can show that, for all $b_1,b_2,b_3\in B$,
	\begin{equation}
		\tau_2^{(\psi,\theta)}(b_2, b_3) \tau_2^{(\psi,\theta)}( b_1, b_2 b_3) = \tau_2^{(\psi,\theta)}(b_1 b_2, b_3) \sigma_{b_3}(\tau_2^{(\psi,\theta)}(b_1, b_2)).\label{z2 condn2}
	\end{equation}
If $a\in A$ and $b_1,b_2\in B$, then we see that
	\begin{eqnarray}
		\rho^{(\psi,\theta)}(\beta_{b_2}(a), b_1)\nu_{b_1}(\rho^{(\psi,\theta)}(a,b_2))&=&\theta_1^{-1}\big(\rho(\psi_1(\beta_{b_2}(a)),\psi_2(b_1))\big)\nu_{b_1}\big(\theta_1^{-1}(\rho(\psi_1(a),\psi_2(b_2)))\big)\nonumber\\
		&=&\theta_1^{-1}\big(\rho(\beta_{\psi_2(b_2)}(\psi_1(a)),\psi_2(b_1))\big)\theta_1^{-1}\big(\nu_{\psi_2(b_1)}(\rho(\psi_1(a),\psi_2(b_2)))\big),\nonumber\\
		&&\text{since }\psi\in \Aut(\mathcal{A})\text{ and from }\eqref{act stb}\nonumber \\
		&=&\theta_1^{-1}\big(\rho(\beta_{\psi_2(b_2)}(\psi_1(a)),\psi_2(b_1))\nu_{\psi_2(b_1)}(\rho(\psi_1(a),\psi_2(b_2)))\big)\nonumber\\
		&=&\theta_1^{-1}\big(\rho(\psi_1(a), \psi_2(b_1)\psi_2( b_2)) \, \nu_{\psi_2(b_1 )\psi_2(b_2)}( f(\tau_2(\psi_2(b_1), \psi_2(b_2)), \psi_1(a)))\big),\nonumber\\
		&&\text{using }\eqref{cocycle3}\nonumber\\
		&=&\theta_1^{-1}\big(\rho(\psi_1(a), \psi_2(b_1 b_2))\big)\nu_{b_1b_2}\big(\theta_1^{-1}(f(\tau_2(\psi_2(b_1), \psi_2(b_2)), \psi_1(a)))\big),\nonumber\\
		&&\text{using }\eqref{act stb}\nonumber\\
		&=&\theta_1^{-1}\big(\rho(\psi_1(a), \psi_2(b_1 b_2))\big)\nu_{b_1b_2}\big(\theta_1^{-1}(f(\theta_2(\theta_2^{-1}(\tau_2(\psi_2(b_1), \psi_2(b_2)))), \psi_1(a)))\big)\nonumber\\
		&=&\theta_1^{-1}\big(\rho(\psi_1(a), \psi_2(b_1 b_2))\big)\nu_{b_1b_2}\big(f(\theta_2^{-1}(\tau_2(\psi_2(b_1), \psi_2(b_2))), a)\big),\nonumber\\
		&&\text{using \eqref{act stb}}\nonumber\\
		&=&\rho^{(\psi,\theta)}(a, b_1 b_2) \, \nu_{b_1 b_2}\big( f(\tau_2^{(\psi,\theta)}(b_1, b_2), a)\big). \label{z2 condn3}
	\end{eqnarray}
Similarly, for $a_1, a_2\in A$ and $b\in B$, we have
	\begin{eqnarray}
		\rho^{(\psi,\theta)}(a_1 a_2, b) \, \nu_{b}(\tau_1^{(\psi,\theta)}(a_1, a_2)) &=&\theta_1^{-1}\big(\rho(\psi_1(a_1a_2),\psi_2(b))\big)\nu_b\big(\theta_1^{-1}(\tau_1(\psi_1(a_1),\psi_1(a_2)))\big)\nonumber\\
		&=&\theta_1^{-1}\big(\rho(\psi_1(a_1)\psi_1(a_2),\psi_2(b))\big)\theta_1^{-1}\big(\nu_{\psi_2(b)}(\tau_1(\psi_1(a_1),\psi_1(a_2)))\big),\nonumber\\
		&&\text{using }\eqref{act stb}\nonumber\\
	&=&\theta_1^{-1}\big(\rho(\psi_1(a_1)\psi_1(a_2),\psi_2(b))\nu_{\psi_2(b)}(\tau_1(\psi_1(a_1),\psi_1(a_2)))\big),\nonumber\\
	&&	\text{from }\eqref{cocycle4}\nonumber\\
		&=&\theta_1^{-1}\big(\mu_{ \beta_{\psi_2(b)}(\psi_1(a_2))}(\rho(\psi_1(a_1),\psi_2(b)))\; \rho(\psi_1(a_2),\psi_2( b))\nonumber \\
		&&  \tau_1(\beta_{\psi_2(b)}(\psi_1(a_1)), \beta_{\psi_2(b)}(\psi_1(a_2)))\big)\nonumber\\
	&=&\mu_{ \beta_{b}(a_2)}\big(\theta_1^{-1}(\rho(\psi_1(a_1),\psi_2(b)))\big)\; \theta_1^{-1}(\rho(\psi_1(a_2),\psi_2( b))) \; \nonumber\\
	&&\theta_1^{-1}( \tau_1(\psi_1(\beta_{b}(a_1)), \psi_1(\beta_{b}(a_2)))),\nonumber\\
		&&\text{using }\eqref{act stb}\nonumber\\
		&=&\mu_{ \beta_{b}(a_2)}(\rho^{(\psi,\theta)}(a_1,b))\; \rho^{(\psi,\theta)}(a_2, b) \;  \tau_1^{(\psi,\theta)}(\beta_{b}(a_1), \beta_{b}(a_2)).\label{z2 condn4}
	\end{eqnarray}
	Observe that,
	\begin{eqnarray*}
		\psi_1(a_1\circ_T a_2)&=&\psi_1(a_1\beta_{T(a_1)}(a_2))\\
		&=&\psi_1(a_1)\psi_1(\beta_{T(a_1)}(a_2))\\
		&=&	\psi_1(a_1)\beta_{\psi_2(T(a_1))}(\psi_1(a_2)),\\
		&&\text{as }(\psi_1,\psi_2)\text{ is an automorphism of }\mathcal{A}\\
	\end{eqnarray*}
	\begin{eqnarray}
		&=&	\psi_1(a_1)\beta_{T(\psi_1(a_1))}(\psi_1(a_2))\nonumber\\
		&=&\psi_1(a_1)\circ_T\psi_1(a_2)\label{psi_1 condn}
	\end{eqnarray}
	for all $a_1,a_2\in A.$ Additionally, we see that
	\begin{eqnarray}
 && \tau_2^{(\psi,\theta)}(T(a_1), T(a_2) ) \delta^1_{\sigma}(\chi^{(\psi,\theta)})(a_1, a_2)\\
 &=&\theta_2^{-1}\big(\tau_2(\psi_2(T(a_1)),\psi_2(T(a_2)))\big)\theta_2^{-1}\big(\chi(\psi_1(a_1))\big)\nonumber\\
		&&\big(\theta_2^{-1}\big(\chi(\psi(a_1\circ_T a_2))\big)\big)^{-1}\sigma_{T(a_2)}\big(\theta_2^{-1}(\chi(\psi_1(a_1)))\big)\nonumber\\
		&=&\theta_2^{-1}\big(\tau_2(T(\psi_1(a_1)),T(\psi_1(a_2)))\chi(\psi_1(a_1))\big)\nonumber\\
		&&\big(\theta_2^{-1}\big(\chi(\psi(a_1)\circ_T\psi_1( a_2))\big)\big)^{-1}\theta_2^{-1}\big(\sigma_{\psi_2(T(a_2))}(\chi(\psi_1(a_1)))\big),\nonumber\\
		&&\text{since }(\psi_1,\psi_2)\in \Aut(\mathcal{A})\text{ and using }\eqref{act stb},\eqref{psi_1 condn}\nonumber\\
		&=&\theta_2^{-1}\big(\tau_2(T(\psi_1(a_1)),T(\psi_1(a_2)))\chi(\psi_1(a_1))\big)\nonumber\\
		&&\big(\theta_2^{-1}\big(\chi(\psi(a_1)\circ_T\psi_1( a_2))\big)\big)^{-1}\theta_2^{-1}\big(\sigma_{T(\psi_1(a_2))}(\chi(\psi_1(a_1)))\big),\nonumber\\
		&&\text{using }\eqref{act stb}\nonumber\\
		&=&\theta_2^{-1}\big(S \big(\nu^{-1}_{T(\psi_1(a_1) \circ_T \psi_1(a_2))}\big(\rho(\psi_1(a_2), T(\psi_1(a_1)))  \,\nonumber \\
		&&\tau_1(\psi_1(a_1),\beta_{T(\psi_1(a_1))}\psi_1((a_2)))  \,\nu_{T(\psi_1(a_1))}(f(\chi(\psi_1(a_1)), \psi_1(a_2)))  \big)\big)\big)\nonumber\\
		&=&S\big(\theta_1^{-1}\big(\nu^{-1}_{T(\psi_1(a_1 \circ_T a_2))}(\rho(\psi_1(a_2), \psi_2(T(a_1))))\big)   \nonumber\\
		&&\theta_1^{-1}\big(\tau_1(\psi_1(a_1),\beta_{\psi_2(T(a_1))}\psi_1((a_2)))\big)\theta_1^{-1}\big(\nu_{\psi_2(T(a_1))}(f(\chi(\psi_1(a_1)), \psi_1(a_2)))\big)\big)\nonumber\\
	&=&S\big(\nu^{-1}_{a_1 \circ_T a_2}\big(\theta_1^{-1}(\rho(\psi_1(a_2), \psi_2(T(a_1))))\big)   \theta_1^{-1}\big(\tau_1(\psi_1(a_1),\psi_1(\beta_{T(a_1)}(a_2)))\big)\nonumber\\
		&&\nu_{T(a_1)}\big(\theta_1^{-1}(f(\theta_2(\theta_2^{-1}(\chi(\psi_1(a_1)))), \psi_1(a_2)))\big)\big)\nonumber\\
		&=&S\big(\nu^{-1}_{a_1 \circ_T a_2}\big(\theta_1^{-1}(\rho(\psi_1(a_2), \psi_2(T(a_1))))\big)   \theta_1^{-1}\big(\tau_1(\psi_1(a_1),\psi_1(\beta_{T(a_1)}(a_2)))\big)\nonumber\\
		&&\nu_{T(a_1)}\big(f(\theta_2^{-1}(\chi(\psi_1(a_1))), \psi_1(a_2))\big)\big)\nonumber\\
		&=&S \big(\nu^{-1}_{T(a_1 \circ_T a_2)}\big(\rho^{(\psi,\theta)}(a_2, T(a_1))  \, \tau_1^{(\psi,\theta)}(a_1,\beta_{T(a_1)}(a_2))\nonumber \\
		&&\,\nu_{T(a_1)}(f(\chi^{(\psi,\theta)}(a_1), a_2))  \big)\big).\label{z2 condn5}
	\end{eqnarray}
	The equations \eqref{z2 condn1},\eqref{z2 condn2},\eqref{z2 condn3},\eqref{z2 condn4} and \eqref{z2 condn5} prove that $\big(\tau_1^{(\psi, \theta)}, \tau_2^{(\psi, \theta)},\rho^{(\psi, \theta)},\chi^{(\psi, \theta)}\big) \in \Z^2_{RRB}(\mathcal{A},\mathcal{K})$, which is assertion (1).
	\par
	
Let $(\tau_1,\tau_2,\rho,\chi)\in \B^2_{RRB}(\mathcal{A},\mathcal{K})$. Then there exists $(\kappa_1,\kappa_2)\in C^{1}_{RRB} = C^1(A, K) \oplus C^1(B,L)$ such that $$\partial_{RRB}^1(\kappa_1, \kappa_2)=(\partial_{\mu}^1(\kappa_1), \partial_{\sigma}^1(\kappa_2), \lambda_1, \lambda_2  )=(\tau_1,\tau_2,\rho,\chi). $$
	We show that $\partial_{RRB}^1(\theta_1^{-1}\kappa_1\psi_1,\theta_2^{-1}\kappa_2\psi_2)=\big(\tau_1^{(\psi, \theta)}, \tau_2^{(\psi, \theta)},\rho^{(\psi, \theta)},\chi^{(\psi, \theta)}\big).$ For $a_1,a_2\in A,$ we have $$\tau_1(a_1,a_2)=\kappa_1(a_2)\kappa_1(a_1a_2)^{-1}\mu_{a_2}(\kappa_1(a_1))$$
and
	\begin{eqnarray}
		\tau_1^{(\psi, \theta)}(a_1, a_2)&=&\theta_1^{-1}\big(\tau_1(\psi_1(a_1), \psi_1(a_2))\big)\nonumber\\
		&=&\theta_1^{-1}\big(\kappa_1(\psi_1(a_2))\kappa_1(\psi_1(a_1)\psi_1(a_2))^{-1}\mu_{\psi_1(a_2)}\big(\kappa_1(\psi_1(a_1))\big)\big)\nonumber\\
		&=&\theta_1^{-1}\big(\kappa_1(\psi_1(a_2))\big)(\theta_1^{-1}\big(\kappa_1(\psi_1(a_1)\psi_1(a_2)))\big)^{-1}\theta_1^{-1}\big(\mu_{\psi_1(a_2)}\big(\kappa_1(\psi_1(a_1))\big)\big)\nonumber
	\end{eqnarray}
\begin{eqnarray}
		&=&\theta_1^{-1}\big(\kappa_1(\psi_1(a_2))\big)(\theta_1^{-1}\big(\kappa_1(\psi_1(a_1a_2)))\big)^{-1}\mu_{a_2}\big(\theta_1^{-1}(\kappa_1(\psi_1(a_1))\big)\big),\nonumber\\
		&&\text{from \eqref{act stb}}\nonumber\\
		&=&\partial_{\mu}(\theta_1^{-1}\kappa_1\psi_1)(a_1,a_2).\label{b2condn1}
	\end{eqnarray}
	Similarly, we can prove that 
	\begin{equation}\label{b2condn2}
\tau_2^{(\psi, \theta)}(b_1, b_2)=\partial_{\sigma}(\theta_2^{-1}\kappa_2\psi_2)(b_1,b_2).
	\end{equation}
	Also, for $a\in A$ and $b\in B$, we compute
	\begin{eqnarray}
		\rho^{(\psi, \theta)}(a,b)&=&\theta_1^{-1}(\rho(\psi_1(a),\psi_2(b)))\nonumber\\
		&=&\theta_1^{-1}\big(\nu_{\psi_2(b)} (f(\kappa_2(\psi_2(b)), \psi_1(a))\kappa_1(\psi_1(a)) )\, (\kappa_1(\beta_{\psi_2(b)}(\psi_1(a))))^{-1}\big)\nonumber\\
		&=&\theta_1^{-1}\big(\nu_{\psi_2(b)} (f(\kappa_2(\psi_2(b)), \psi_1(a))\kappa_1(\psi_1(a)) )\big)\,
		\big(	\theta_1^{-1}(\kappa_1(\psi_1(\beta_b(a))))\big)^{-1},\nonumber\\
		&&\text{as $(\psi_1,\psi_2)\in \Aut(\mathcal{A})$ and $\theta_1$ is a group automorphism of }K\nonumber\\
		&=&\theta_1^{-1}\big(\theta_1\nu_{b}\theta_1^{-1} (f(\kappa_2(\psi_2(b)), \psi_1(a))\kappa_1(\psi_1(a)) )\big)\,
		\big(	\theta_1^{-1}(\kappa_1(\psi_1(\beta_b(a))))\big)^{-1},\nonumber\\
		&&\text{using \eqref{act stb}}\nonumber\\
		&=&\nu_{b}\big(\theta_1^{-1} (f(\kappa_2(\psi_2(b)), \psi_1(a)))\theta_1^{-1}(\kappa_1(\psi_1(a)) )\big)\,
		\big(	\theta_1^{-1}(\kappa_1(\psi_1(\beta_b(a))))\big)^{-1}\nonumber\\
		&=&\nu_{b}\big(\theta_1^{-1} (f(\theta_2\theta_2^{-1}\kappa_2(\psi_2(b)), \psi_1(a)))\theta_1^{-1}(\kappa_1(\psi_1(a)) )\big)\,
		\big(	\theta_1^{-1}(\kappa_1(\psi_1(\beta_b(a))))\big)^{-1}\nonumber\\
		&=&\nu_{b}\big(f(\theta_2^{-1}\kappa_2(\psi_2(b)), a)\theta_1^{-1}(\kappa_1(\psi_1(a)) )\big)\,
		\big(	\theta_1^{-1}(\kappa_1(\psi_1(\beta_b(a))))\big)^{-1},\label{b2condn3}\\
		&&\text{using }\eqref{act stb}.\nonumber
	\end{eqnarray}
Similarly, for $a\in A$, we observe
	\begin{eqnarray}
		\chi^{(\psi, \theta)}(a)&=&\theta_2^{-1}(\chi(\psi_1(a)))\nonumber\\
		&=&\theta_2^{-1}\big( S \big(\nu^{-1}_{T(\psi_1(a))}(\kappa_1(\psi_1(a))) \big)\big) \, \theta_2^{-1}\big(\big(\kappa_2(T(\psi_1(a))) \big)^{-1}\big)\nonumber\\
		&=&S\big( \theta_1^{-1} \big(\nu^{-1}_{T(\psi_1(a))}(\kappa_1(\psi_1(a))) \big)\big) \, \theta_2^{-1}\big(\big(\kappa_2(T(\psi_1(a))) \big)^{-1}\big)\nonumber\\
		&=&S\big( \theta_1^{-1} \big(\theta_1\nu^{-1}_{T(a)}\theta_1^{-1}(\kappa_1(\psi_1(a))) \big)\big) \, \big(\theta_2^{-1}\kappa_2\psi_2(T(a))\big)^{-1},\nonumber\\
		&&\text{using }\eqref{act stb}\nonumber\\
		&=&S\big( \nu^{-1}_{T(a)}(\theta_1^{-1}\kappa_1\psi_1(a)) \big) \, \big(\theta_2^{-1}\kappa_2\psi_2(T(a))\big)^{-1}.\label{b2condn4}	\end{eqnarray}
Hence, the equations \eqref{b2condn1},\eqref{b2condn2},\eqref{b2condn3} and \eqref{b2condn4} prove that $\big(\tau_1^{(\psi, \theta)}, \tau_2^{(\psi, \theta)},\rho^{(\psi, \theta)},\chi^{(\psi, \theta)}\big) \in \B^2_{RRB}(\mathcal{A},\mathcal{K})$, which is assertion (2).
\end{proof}

\begin{remark}
 The action of $\C_{(\nu,\mu,\sigma,f)}$ on $\Ho_{RRB}^2(\mathcal{A}, \mathcal{K})$, as outlined in the preceding lemma, can be transferred on $\Ext_{(\nu,\mu,\sigma,f)}(\mathcal{A}, \mathcal{K})$ via the bijection established in Theorem \eqref{ext and cohom bijection}. In fact, the action of $\C_{(\nu,\mu,\sigma,f)}$ on $\Ext_{(\nu,\mu,\sigma,f)}(\mathcal{A}, \mathcal{K})$ corresponds precisely to the action defined in \eqref{act1 RRB}.
\end{remark}

We now consider the action of $\Ho_{RRB}^2(\mathcal{A}, \mathcal{K})$ on itself through right translation, which is obviously faithful and transitive. Once again employing Theorem \ref{ext and cohom bijection}, we can transfer this action on $\Ext_{(\nu,\mu,\sigma,f)}(\mathcal{A}, \mathcal{K})=\{ [\mathcal{E}(\tau_1,\tau_2,\rho,\chi)] \mid (\tau_1,\tau_2,\rho,\chi) \in \Z^2_{RRB}(\mathcal{A},\mathcal{K})\}$. More precisely, for $[(\tau^\prime_1,\tau^\prime_2,\rho^\prime,\chi^\prime)] \in \Ho_{RRB}^2(\mathcal{A},\mathcal{K})$, the action is defined as $$[\mathcal{E}(\tau_1,\tau_2,\rho,\chi)]^{[(\tau^\prime_1,\tau^\prime_2,\rho^\prime,\chi^\prime)]} = [\mathcal{E}(\tau_1 \tau^\prime_1,\tau_2 \tau^\prime_2,\rho \rho^\prime, \chi \chi^\prime)]$$ for all $[\mathcal{E}(\tau_1,\tau_2,\rho,\chi)] \in \Ext_{(\nu,\mu,\sigma,f)}(\mathcal{A}, \mathcal{K})$. Notably, this action is also faithful and transitive.

Consider the semidirect product $\Gamma:= \C_{(\nu,\mu,\sigma,f)} \ltimes \Ho_{RRB}^2(\mathcal{A}, \mathcal{K})$ under the action defined in Lemma \ref{lemma-act3 RRB}.  Our next step is to define an action of $\Gamma$  on $\Ext_{(\nu,\mu,\sigma,f)}(\mathcal{A}, \mathcal{K})$. For $(c, h) \in  \Gamma$ and $[\mathcal{E}] \in \Ext_{(\nu,\mu,\sigma,f)}(\mathcal{A}, \mathcal{K})$, we define
\begin{align}\label{act 4 RRB}
[\mathcal{E}]^{(c,h)}:= ([\mathcal{E}]^c)^h.
\end{align}

\begin{lemma}
The rule defined in \eqref{act 4 RRB} gives an action of the group $\Gamma$ on the set $\Ext_{(\nu,\mu,\sigma,f)}(\mathcal{A}, \mathcal{K})$.
\end{lemma}

\begin{proof}
For $(c_1,h_1),  (c_2, h_2) \in \Gamma$,  we have $(c_1,h_1)(c_2,h_2)=(c_1c_2, h_1^{c_2} \, h_2)$.  So, it is enough to show that $\big([\mathcal{E}]^h\big)^c = \big([\mathcal{E}]^c\big)^{h^c}$ for each $c \in \C_{(\nu, \mu, \sigma,f)}$, $h \in \Ho^2_{RRB}(\mathcal{A},\mathcal{K})$ and $[\mathcal{E}] \in \Ext_{(\nu, \mu, \sigma,f)}(\mathcal{A},\mathcal{K})$.
	 Then, for $h=[(\tau^\prime_1,\tau^\prime_2,\rho^\prime,\chi^\prime)] \in \Ho^2_{RRB}(\mathcal{A}, \mathcal{K})$, we have
	\begin{eqnarray*}
		\big([\mathcal{E}]^{h}\big)^c &=& [\mathcal{E}^c(\tau_1 \tau^\prime_1,\tau_2 \tau^\prime_2,\rho \rho^\prime, \chi \chi^\prime)]\\
		&=&[\mathcal{E}(\tau_1^c \tau^{\prime c}_1,\tau_2^c \tau^{\prime c}e_2,\rho^c \rho^{\prime c}, \chi^c \chi^{\prime c})]\\
		&=& ([\mathcal{E}(\tau_1^c,\tau_2^c,\rho^c,\chi^c)])^{h^c}\\
		&=& \big([\mathcal{E}]^c\big)^{h^c},
	\end{eqnarray*}
which is desired.
\end{proof}
 Let $[\mathcal{E}] \in \Ext_{(\nu, \mu, \sigma,f)}(\mathcal{A},\mathcal{K})$ be a fixed abelian extension. Since the action of $\Ho^2_{RRB}(\mathcal{A},\mathcal{K})$ on $\Ext_{(\nu, \mu, \sigma,f)}(\mathcal{A},\mathcal{K})$ is transitive and faithful, for  each $c \in \C_{(\nu, \mu, \sigma,f)}$, there exists a unique element (say) $h_c$  in  $\Ho^2_{RRB}(\mathcal{A},\mathcal{K})$ such that  
$$[\mathcal{E}]^{c} = [\mathcal{E}]^{h_c}.$$
We thus have a well-defined map $ \omega(\mathcal{E}): \C_{(\nu, \mu, \sigma,f)} \rightarrow \Ho^2_{RRB}(\mathcal{A},\mathcal{K})$ given by
\begin{equation}\label{wells-map RRB}
	\omega(\mathcal{E})(c):=h_c
\end{equation}
for $c \in \C_{(\nu, \mu, \sigma,f)}$. 

\begin{lemma}\label{wells3 RRB}
	The map $ \omega(\mathcal{E}): \C_{(\nu, \mu, \sigma,f)} \rightarrow \Ho^2_{RRB}(\mathcal{A},\mathcal{K})$ given in \eqref{wells-map RRB} is a derivation with respect to the action of $\C_{(\nu, \mu, \sigma,f)}$ on $\Ho^2_{RRB}(\mathcal{A},\mathcal{K})$ given in \eqref{act3 RRB}.
\end{lemma}
\begin{proof}
	Let $c_1, c_2 \in \C_{(\nu, \mu,  \sigma,f)}$ and  $\omega(\mathcal{E})(c_1c_2) = h_{c_1c_2}$.  Thus, by the definition of $\omega(\mathcal{E})$,  we have $[\mathcal{E}]^{c_1c_2} = [\mathcal{E}]^{h_{c_1c_2}}$.  Using the fact that  $\big([\mathcal{E}]^{h}\big)^{c} = \big([\mathcal{E}]^{c}\big)^{h^c}$ for each $c \in \C_{(\nu, \mu, \sigma,f)}$, $h \in \Ho^2_{RRB}(\mathcal{A},\mathcal{K})$, we have
	\begin{eqnarray*}
		[\mathcal{E}]^{h_{c_1c_2}} & = &[\mathcal{E}]^{(c_1c_2)}\\
		&=& \big([\mathcal{E}]^{c_1}\big)^{c_2}\\
		&=& \big([\mathcal{E}]^{h_{c_1}}\big)^{c_2}\\
		&= & \big([\mathcal{E}]^{c_2}\big)^{(h_{c_1})^{c_2}}\\
		&=& \big([\mathcal{E}]^{h_{c_2}}\big)^{(h_{c_1})^{c_2}}\\
		&=& [\mathcal{E}]^{\big(h_{c_2} (h_{c_1})^{c_2} \big)}.
	\end{eqnarray*}
	Since the action of $\Ho^2_{RRB}(\mathcal{A},\mathcal{K})$ on $\Ext_{(\nu, \mu, \sigma,f)}(\mathcal{A},\mathcal{K})$  is faithful,  it follows that $h_{c_1c_2} = (h_{c_1})^{c_2}  h_{c_2}$. This implies that $\omega(\mathcal{E})(c_1c_2)=\big(\omega(\mathcal{E})(c_1)\big)^{c_2}\omega(\mathcal{E})(c_2)$, and hence $\omega(\mathcal{E})$ is a derivation. 
\end{proof}
Let 
$$\mathcal{E} : \quad  {\bf 1} \longrightarrow \mathcal{K}=(K,L, \alpha,S ) \stackrel{(i_1, i_2)}{\longrightarrow}  \mathcal{H}=(H,G, \phi, R) \stackrel{(\pi_1, \pi_2)}{\longrightarrow}\mathcal{A}= (A,B, \beta, T) \longrightarrow {\bf 1},$$
be an abelian extension of a relative Rota--Baxter group $\mathcal{A}$ by a trivial relative Rota--Baxter group $\mathcal{K}$ such that $[\mathcal{E}] \in \Ext_{(\nu, \mu, \sigma,f)}(\mathcal{A},\mathcal{K})$.
Let  $\Aut_{\mathcal{K}}(\mathcal{H})$ denote the subgroup of $\Aut(\mathcal{H})$ consisting of all automorphisms of $\mathcal{H}$ which normalize $\mathcal{K}$, that is,
$$\Aut_{\mathcal{K}}(\mathcal{H}) := \{ \gamma=(\gamma_1,\gamma_2) \in \Aut(\mathcal{H}) \mid \gamma_1(k) \in K \mbox{ for all }  k \in K\mbox{ and } \gamma_2(l) \in L \mbox{ for all }  l \in L\}.$$ 
For $\gamma \in \Aut_{\mathcal{K}}(\mathcal{H})$, we set $\gamma_{\mathcal{K}} := (\gamma_1 |_K,\gamma_2|_L)$, the restriction of $\gamma$ to $\mathcal{K}$, and $\gamma_{\mathcal{A}}=(\gamma_{1_A},\gamma_{2_B})$ to be the automorphism of $\mathcal{A}$ induced by $\gamma$. More precisely, $\gamma_{1_A}(a) = \pi_1(\gamma_1(s_H(a)))$ for all $a \in A$ and $\gamma_{2_B}(b) = \pi_2(\gamma_2(s_G(b)))$ for all $b\in B$, where $(s_H,s_G)$ is a set-theoretic section of $(\pi_1,\pi_2)$. Note that, the definition of $\gamma_\mathcal{A}$ is independent of  the choice of a set-theoretic section. Define a map $\rho(\mathcal{E}) :  \Aut_{\mathcal{K}}(\mathcal{H}) \rightarrow  \Aut(\mathcal{A}) \times \Aut(\mathcal{K})$ by
$$\rho(\mathcal{E})(\gamma)=(\gamma_{\mathcal{A}}, \gamma_\mathcal{K}).$$ 
Although $\omega(\mathcal{E})$ is not a homomorphism, but we can still talk about its set-theoretic kernel, that is,
$$\Ker(\omega(\mathcal{E})) = \{c \in C_{(\nu,\mu,\sigma,f)} \mid [\mathcal{E}]^c=[\mathcal{E}]\}.$$

\begin{prop}\label{wells4 RRB}
If $\mathcal{E}$ is an extension with induced action $(\nu, \mu, \sigma,f)$, then  	$\IM(\rho(\mathcal{E})) = \Ker(\omega(\mathcal{E}))$.
\end{prop}

\begin{proof}
First, we show that $\IM(\rho(\mathcal{E}))\subseteq C_{(\nu,\mu,\sigma,f)}$. 
For $\gamma \in \Aut_{\mathcal{K}}(\mathcal{H})$, we have $\rho(\mathcal{E})=(\gamma_{\mathcal{A}}, \gamma_\mathcal{K}).$ We are required  to show that $\nu_{b} = \gamma_1|_K^{-1} \nu_{\gamma_{2_B}(b)} \gamma_1|_K$, $\mu_{a} = \gamma_1|_K^{-1} \mu_{\gamma_{1_A}(a)} \gamma_1|_K$,  $\sigma_{b} = \gamma_2|_L^{-1}\sigma_{\gamma_{2_B}(b)} \gamma_2|_L$ and $\gamma_1|_K(f(l,a))=f(\gamma_2|_L(l),\gamma_{1_A}(a))$ for all $a \in A, b\in B$. Let $(s_H,s_G)$ be an st-section of $(\pi_1,\pi_2)$  and elements $ h\in H$, $g\in G$. Notice that $\gamma_1|_K^{-1}$ is the restriction of $\gamma_1^{-1}$ on $K$ and $\gamma_2|_L^{-1}$ is the restriction of $\gamma_2^{-1}$ on $L$. Also, for $g \in G$,  $s_G(\pi_2(g))=g l_b$ for some $l_g \in L.$ We have
\begin{eqnarray*}
\gamma_1|_K^{-1} \nu_{\gamma_{2_B}(b)} \gamma_1|_K(k)&=& \gamma_1|_K^{-1} \nu_{\gamma_{2_B}(b)}( \gamma_1(k))\\
&=& \gamma_1^{-1}( \nu_{	\pi_2(\gamma_2(s_G(b)))}( \gamma_1(k)))\\
&=& \gamma_1^{-1}( \phi_{s_G(\pi_2(\gamma_2(s_G(b))))}( \gamma_1(k)))\\
&=& \gamma_1^{-1}(\phi_{\gamma_2(s_G(b)) l_{\gamma_2(s_G(b))}}(\gamma_{1}(k)))\\
&=& \gamma_1^{-1}(\phi_{\gamma_2(s_G(b))}(\gamma_{1}(k))) \\
&=& \phi_{s_G(b)}(k)   \mbox{, since $(\gamma_{1},\gamma_2)\in \Aut(\mathcal{H})$ }\\
&=& \nu_b(k).
\end{eqnarray*}	
Similarly, we have $\mu_{a} = \gamma_1|_K^{-1} \mu_{\gamma_{1_A}(a)} \gamma_1|_K$ and  $\sigma_{b} = \gamma_2|_L^{-1}\sigma_{\gamma_{2_B}(b)} \gamma_2|_L$ for all $a \in A$ and $b \in B$. Next we prove that $\gamma_1|_K(f(l,a))=f(\gamma_2|_L(l),\gamma_{1_A}(a))$ for all $a \in A$ and $b\in B$.

\begin{eqnarray}
\gamma_1(f(l,a))&=& \gamma_1(s_H(a)^{-1} \phi_{l}(s_H(a))) \nonumber\\
&=& \gamma_1(s_H(a)^{-1} ) \gamma_1(\phi_{l}(s_H(a))) \nonumber\\
&=& \gamma_1(s_H(a)^{-1} )\phi_{\gamma_{2}(l)}(\gamma_{1}(s_H(a))). \label{fcom}
\end{eqnarray}
Notice that $\gamma_1(s_H(a))=s_H(\gamma_{1_A}(a)) k_a$ for some $k_a \in K$. Using this in \eqref{fcom}, we have
\begin{align*}
	\gamma_1(f(l,a))=& k^{-1}_a s_H( \gamma_{1_A}(a) ^{-1}) \phi_{\gamma_{2}(l)}(s_H(\gamma_{1_A}(a)) k_a)\\
	=& k^{-1}_a f (\gamma_2|_L(l),\gamma_{1_A}(a)) k_a\\
	=& f (\gamma_2|_L(l),\gamma_{1_A}(a)).
\end{align*}

	Now, let $\rho(\mathcal{E})(\gamma) =  (\gamma_{\mathcal{A}}, \gamma_{\mathcal{K}})$ for  $\gamma \in \Aut_{\mathcal{K}}(\mathcal{H})$.  We know that $s_H(\pi_1(x))=xk_x $ for some $k_x \in K$ and $s_G(\pi_2(y))=yl_y $ for some $l_y \in L$.  Thus we have
	\begin{eqnarray*}
		\gamma_{1_A}^{-1}(\pi_1(\gamma_1(s_H(a)))) &=& \pi_1\big(\gamma_{1}^{-1}(s_H(\pi_1(\gamma_1(s_H(a)))))\big)\\
		&=& \pi_1\big(\gamma_{1}^{-1}\big(\gamma_1(s_H(a))y_{\gamma_1(s_H(a))}\big)\big)\\
		&= & a
	\end{eqnarray*}
for all $a\in A$. Similarly, it can be shown that $	\gamma_{2_B}^{-1}(\pi_2(\gamma_2(s_G(b))))=b$ for all $b\in B$, which implies that the diagram
$$\begin{CD}
	\textbf{1} @>>> \mathcal{K} @>>> \mathcal{H}@>{{(\pi_1,\pi_2)} }>> \mathcal{A} @>>> \textbf{1}\\ 
	&& @V{\text{(Id,Id)}} VV @V{\gamma} VV @V{\text{(Id,Id)} }VV \\
	\textbf{1} @>>> \mathcal{K} @>{\gamma_{\mathcal{K}}}>> \mathcal{H} @>{(\gamma_{1_A}^{-1} \pi_1,\gamma_{2_B}^{-1} \pi_2)}>> \mathcal{A}  @>>> 	\textbf{1}
\end{CD}$$
commutes. Hence, we have $[(\mathcal{E})]^{(\gamma_{\mathcal{A}}, \gamma_{\mathcal{K}})} =[(\mathcal{E})]$, which shows that $\IM(\rho(\mathcal{E})) \subseteq \Ker(\omega(\mathcal{E}))$.
Conversely, if $(\psi, \theta) \in  \Ker(\omega(\mathcal{E}))$,  then there exists a relative Rota--Baxter group homomorphism $\gamma: \mathcal{H} \rightarrow \mathcal{H}$ such that the diagram 
$$\begin{CD}
	\textbf{1} @>>> \mathcal{K} @>>> \mathcal{H}@>{{(\pi_1,\pi_2)} }>> \mathcal{A} @>>> 	\textbf{1}\\ 
	&& @V{\text{(Id,Id)}} VV@V{\gamma} VV @V{\text{(Id,Id)} }VV \\
		\textbf{1} @>>> \mathcal{K} @>{\theta}>> \mathcal{H} @>{ (\psi_1^{-1} \pi_1,\psi_2^{-1} \pi_2)}>> \mathcal{A}  @>>> 	\textbf{1}
\end{CD}$$
commutes. It is now obvious that $\gamma \in \Aut(\mathcal{H})$, $\theta=(\theta_1,\theta_2) = (\gamma_1|_K,\gamma_2|_L)$ and $\psi=(\psi_1,\psi_2) = (\gamma_{1_A},\gamma_{2_B})$. Hence, we have $\rho(\mathcal{E})(\gamma)=(\psi, \theta)$, which completes the proof.  
\end{proof}

To proceed further, we set $$\Aut^{\mathcal{A}, \mathcal{K}}(\mathcal{H}) := \{\gamma \in \Aut_{\mathcal{K}}(\mathcal{H}) \mid \gamma_{\mathcal{K}}= (\Id,\Id) ~\textrm{and}~  \gamma_{\mathcal{A}}= (\Id,\Id)\}.$$ Note that,  $\Aut^{\mathcal{A}, \mathcal{K}}(\mathcal{H})$ is precisely the kernel of $\rho(\mathcal{E})$. Hence, using Proposition \ref{wells4 RRB}, we get the following result.

\begin{thm}\label{wells5 RRB}
	Let $\mathcal{E}: 	\textbf{1} \rightarrow \mathcal{K} \rightarrow \mathcal{H} \overset{(\pi_1,\pi_2)}\rightarrow \mathcal{A}\rightarrow 	\textbf{1}$ be an abelian extension of a relative Rota--Baxter group $\mathcal{A}$ by a trivial relative Rota--Baxter group $\mathcal{K}$ such that $[\mathcal{E}] \in \Ext_{(\nu, \mu, \sigma,f)}(\mathcal{A},\mathcal{K})$. Then there is an exact sequence of groups 
	$$1 \rightarrow \Aut^{\mathcal{A},\mathcal{K}}(\mathcal{H}) \rightarrow \Aut_{\mathcal{K}}(\mathcal{H}) \stackrel{\rho(\mathcal{E})}{\longrightarrow} \C_{(\nu, \mu,\sigma,f)} \stackrel{\omega(\mathcal{E})}{\longrightarrow} \Ho^2_{RRB}(\mathcal{A},\mathcal{K}),$$
	where $\omega(\mathcal{E})$ is, in general,  only a derivation.
\end{thm}

We further prove

\begin{thm}\label{wells6 RRB}
	Let $\mathcal{E}: 	\textbf{1} \rightarrow \mathcal{K} \rightarrow \mathcal{H} \overset{(\pi_1,\pi_2)}\rightarrow \mathcal{A}\rightarrow 	\textbf{1}$ be an extension of $\mathcal{A}$ by $\mathcal{K}$ such that $[\mathcal{E}] \in \Ext_{(\nu, \mu, \sigma,f)}(\mathcal{A}, \mathcal{K})$.  Then  $\Aut^{\mathcal{A}, \mathcal{K}}(\mathcal{H}) \cong \Z^1_{RRB}(\mathcal{A},\mathcal{K})$.
\end{thm}
\begin{proof}
	We know that every element $h\in H$ has a unique expression of the form $h=s_H(a)k$ for some $a\in A$ and $k\in K$. Similarly, every element $g\in G$ has a unique expression of the form $g=s_G(b)l$ for some $b\in B$ and $l\in L$. 	We define a map  $\eta: \Z^1_{RRB}(\mathcal{A},\mathcal{K})  \rightarrow  \Aut^{\mathcal{A},\mathcal{K}}(\mathcal{H})$ by 
	$$\eta(\kappa_1, \kappa_2) = (\kappa_1^\eta, \kappa_2^\eta),$$
where $(\kappa_1,\kappa_2)\in\Z^1_{RRB}(\mathcal{A},\mathcal{K})$, $\kappa_1^\eta(s_H(a)k)= s_H(a) \kappa_1(a)k$ and $\kappa_2^\eta( s_G(b)l)=s_G(b) \kappa_2(b)l.$  The image $\eta(\kappa_1,\kappa_2)$ is independent of the choice of a set-theoretic section. To see that the defined map $\eta$ is well-defined, we first need to show that $(\kappa_1^\eta, \kappa_2^\eta) \in \Aut^{\mathcal{A},\mathcal{K}}(\mathcal{H})$. For $a_1,a_2\in A$ and $k_1,k_2\in K,$ we have 
	\begin{eqnarray}
		\kappa_1^{\eta}(s_H(a_1)k_1~s_H(a_2)k_2)&=&	\kappa_1^{\eta}(s_H(a_1 a_2) \tau_1(a_1, a_2) \mu_{a_2}(k_1) k_2)\nonumber\\
		&=& s_H(a_1a_2) \kappa_1(a_1 a_2)\tau_1(a_1, a_2) \mu_{a_2}(k_1) k_2 \nonumber \\
		&=&  s_H(a_1a_2) \kappa_1(a_2) \mu_{a_2}(\kappa(a_1)) \tau_1(a_1, a_2) \mu_{a_2}(k_1) k_2, ~ \mbox{ using \eqref{der1}.} \label{aut1}
	\end{eqnarray}
On the other hand 
\begin{eqnarray}
\kappa_1^{\eta}(s_H(a_1)k_1) ~\kappa_1^{\eta}(s_H(a_2)k_2)&=& \big(s_H(a_1) \kappa_1(a_1)k_1 \big) \big(s_H(a_2) \kappa_1(a_2)k_2 \big) \nonumber\\
&=& s_H(a_1) s_H(a_2) \mu_{a_2}(\kappa_1(a_1)k_1 ) \kappa_1(a_2)k_2 \nonumber\\
&=& s_H(a_1 a_2) \tau_1(a_1, a_2) \mu_{a_2}(\kappa_1(a_1)k_1 ) \kappa_1(a_2)k_2.\label{aut2}
\end{eqnarray}
Since  $K$ is abelian, it is easy to see that the right-hand sides of \refeq{aut1} and \refeq{aut2} are the same. Hence, this shows that $\kappa_1^{\eta}$ is a group homomorphism. Additionally, it is easy to see that $\kappa_1^{\eta}$ is bijective, making it an automorphism of $H$. Similarly, we can show that $\kappa_2^{\eta}$ is an automorphism of $G$. 

Now, we need to show that $(\kappa_1^\eta, \kappa_2^\eta)$ is a morphism of relative Rota--Baxter groups, that is, 
$$ \kappa_2^{\eta}\; R=R \; \kappa_1^{\eta} \hspace{.2cm} \mbox{ and } \hspace{.2cm} \kappa_1^\eta\; \phi_g= \phi_{\kappa_2^\eta(g)} \;  \kappa_1^\eta. $$
For $a \in A$ and $k \in K,$ we have
\begin{eqnarray}
 \kappa_2^{\eta}\; (R \big(s_H(a)k)\big) &= & \kappa_2^{\eta}\big(s_G(T(a)) \chi(a) S(\nu^{-1}_{T(a)}(k))\big) \nonumber \\
 &=& s_G(T(a))  \kappa_2(T(a)) \chi(a) S(\nu^{-1}_{T(a)}(k)). \label{morcom1}
\end{eqnarray}
We also have 
\begin{eqnarray}
R \big( \kappa_1^{\eta}(s_H(a) k ) \big) &=& R \big( s_H(a) \kappa_1(a) k \big)\nonumber \\
&=& s_G(T(a)) \chi(a) S(\nu^{-1}_{T(a)}( \kappa_1(a) k )). \label{morcom2}
\end{eqnarray}
Using \eqref{der4}, we conclude that \eqref{morcom1} and \eqref{morcom2} are equal. Hence, $ \kappa_2^{\eta}\; R=R \; \kappa_1^{\eta}.$
\par
For $b \in B$ and $ l \in L$, we have 
\begin{eqnarray}
\kappa_1^\eta \big(\phi_{s_G(b) l} (s_H(a)k) \big) &=& \kappa_1^\eta \big( s_H(\beta_{b}(a)) \rho(a,b) \nu_b(f(l,a)k) \big)\nonumber\\
&=&  s_H(\beta_{b}(a)) \kappa_1(\beta_b(a)) \rho(a,b) \nu_b(f(l,a)k)  \label{actcom1}
\end{eqnarray}
and
\begin{eqnarray}
\phi_{\kappa_2^\eta(s_G(b) l)}  \big(\kappa_1^\eta(s_H(a) k)\big) &=& \phi_{(s_G(b)\kappa_2(b) l)} (s_H(a) \kappa_1(a) k) \nonumber \\
&=& s_H(\beta_{b}(a)) \rho(a,b) \nu_b(f(\kappa_2(b) l, a)  \kappa_1(a) k).  \label{actcom2}
\end{eqnarray}
By employing equation \eqref{der3} and Lemma \eqref{properties of f}, it follows that the right-hand sides of both \eqref{actcom1} and \eqref{actcom2} coincide. Consequently, we have  $\kappa_1^\eta \; \phi_g= \phi_{\kappa_2^\eta(g)} \; \kappa_1^\eta$. This  establishes the well-definedness of $\eta$. Furthermore, it is easy to see that $\eta$ is a group homomorphism.
\par 

We will now define a map $\zeta :  \Aut^{\mathcal{A},\mathcal{K}}(\mathcal{H})  \rightarrow \Z^1_{RRB}(\mathcal{A},\mathcal{K})$. Let $\gamma =(\gamma_1,\gamma_2)\in \Aut^{\mathcal{A},\mathcal{K}}(\mathcal{H})$. Then, for each  $a \in A$ and $b\in B$, there exists a unique element (say) $y^{\gamma_1}_a \in K$ and $y^{\gamma_2}_b \in L$ such that  $\gamma_1(s_H(a)) = s_H(a) y^{\gamma_1}_a$  and $\gamma_2(s_G(b)) = s_G(b) y^{\gamma_2}_b$. Thus, for $a\in A$ and $b\in B$, we can define $\zeta$ by

$$\zeta(\gamma_1, \gamma_{2}):=(\gamma_1^\zeta, \gamma^{\zeta}_{2}),$$

\noindent where $\gamma_1^\zeta(a)= y^{\gamma_1}_a \mbox{ and } \gamma^{\zeta}_{2}(b)= y^{\gamma_2}_b.$ To establish the well-definedness of $\zeta$, we must demonstrate that the pair $(\gamma_1^\zeta, \gamma_2^\zeta)$ satisfies the conditions \eqref{der1}, \eqref{der2}, \eqref{der3}, and \eqref{der4}. Conditions \eqref{der1} and \eqref{der2} are the standard requirements for group derivations, well-documented in existing group theory literature (see \cite{W71}). We verify that $(\gamma_1^\zeta, \gamma_2^\zeta)$ fulfills \eqref{der3} and \eqref{der4}. For $a \in A$ and $b \in B$, we have 
\begin{eqnarray*}
\gamma_1 \big( \phi_{s_G(b)} (s_H(a)) \big) &=& \gamma_1 ( s_H(\beta_b(a)) \rho(a,b) )\\
 \phi_{\gamma_2(s_G(b))} \big( \gamma_1(s_H(a)) \big)&=&  \gamma_1 ( s_H(\beta_b(a))) \rho(a,b) \\
 \phi_{s_G(b) y^{\gamma_2}_b}(s_H(a) y^{\gamma_1}_a)  &=&  s_H(\beta_b(a)) y^{\gamma_1}_{\beta_b(a)}   \rho(a,b)\\
s_H(\beta_b(a)) \rho(a,b) \nu_b(f(y^{\gamma_2}_b, a ) y^{\gamma_1}_a ) &=& s_H(\beta_b(a)) y^{\gamma_1}_{\beta_b(a)}   \rho(a,b)~\mbox{, by \eqref{nuact}}\\
\nu_b(f(y^{\gamma_2}_b, a ) y^{\gamma_1}_a ) &=& y^{\gamma_1}_{\beta_b(a)}.
\end{eqnarray*}
This shows that $(\gamma_1^\zeta, \gamma^{\zeta}_{2})$ satisfies \eqref{der3}. Next, we have
\begin{eqnarray*}
R  \big(\gamma_{1} (s_H(a)) \big) &=& \gamma_{2} \big(R(s_H(a)) \big)\\
R(s_H(a) y^{\gamma_1}_a) &=& \gamma_{2} \big(s_G(T(a))  \chi(a)  \big)\\
s_G(T(a)) \chi(a) S(\nu^{-1}_{T(a)}(y^{\gamma_1}_a)) &=& s_G(T(a))  \chi(a)  y^{\gamma_2}_{T(a)}\\
 S(\nu^{-1}_{T(a)}(y^{\gamma_1}_a))  &=& y^{\gamma_2}_{T(a)},
\end{eqnarray*}
which shows that $(\gamma_1^\zeta, \gamma_2^\zeta)$ satisfies  \eqref{der4}. It is easy to see that $\zeta :  \Aut^{\mathcal{A},\mathcal{K}}(\mathcal{H})  \rightarrow \Z^1_{RRB}(\mathcal{A},\mathcal{K})$ is a group homomorphism, and that $\eta$ and $\zeta$ are inverses of each other. This proves the theorem. 
\end{proof}
We finally get the following Wells-like exact sequence for relative Rota--Baxter groups.
\begin{thm}\label{wells7 RRB}
	Let $\mathcal{E}:	\textbf{1} \rightarrow \mathcal{K} \rightarrow \mathcal{H} \overset{(\pi_1,\pi_2)}\rightarrow \mathcal{A}\rightarrow 	\textbf{1}$ be an abelian extension of a relative Rota--Baxter group $\mathcal{A}$ by a trivial  relative Rota--Baxter group $\mathcal{K}$ such that $[\mathcal{E}] \in \Ext_{(\nu, \mu, \sigma,f)}(\mathcal{A},\mathcal{K})$. Then there is an exact sequence of groups 
	$$1 \rightarrow \Z^1_{RRB}(\mathcal{A},\mathcal{K}) \rightarrow \Aut_{\mathcal{K}}(\mathcal{H}) \stackrel{\rho(\mathcal{E})}{\longrightarrow} \C_{(\nu, \mu,\sigma,f)} \stackrel{\omega(\mathcal{E})}{\longrightarrow} \Ho^2_{RRB}(\mathcal{A},\mathcal{K}),$$
	where $\omega(\mathcal{E})$ is, in general,  only a derivation. 
\end{thm}

A pair of  automorphisms $(\psi, \theta) \in \Aut(\mathcal{A}) \times \Aut(\mathcal{K})$ is said to be \emph{inducible} if $(\psi, \theta) \in \im(\rho(\mathcal{E}))$. As a straightforward application of Theorem \ref{wells7 RRB}, we get
\begin{cor}
	Let $\mathcal{E}: 	\textbf{1} \rightarrow \mathcal{K} \rightarrow \mathcal{H} \overset{(\pi_1,\pi_2)}\rightarrow \mathcal{A}\rightarrow 	\textbf{1}$ be an abelian extension of a relative Rota--Baxter goup $\mathcal{A}$ by a trivial  relative Rota--Baxter group $\mathcal{K}$ such that $[\mathcal{E}] \in \Ext_{(\nu, \mu, \sigma,f)}(\mathcal{A},\mathcal{K})$. If  $\Ho^2_{RRB}(\mathcal{A},\mathcal{K})$ is the trivial group, then  every element of $\C_{(\nu, \mu,\sigma,f)}$ is inducible. 
\end{cor}
Let $\mathcal{E}: 	\textbf{1} \rightarrow \mathcal{K} \rightarrow \mathcal{H} \overset{(\pi_1,\pi_2)}\rightarrow \mathcal{A}\rightarrow 	\textbf{1}$ be an abelian extension of a relative Rota--Baxter group $\mathcal{A}$ by a trivial relative Rota--Baxter group $\mathcal{K}$. Then $\mathcal{K}$ can be viewed as an $\mathcal{A}$-module through the corresponding quadruplet of actions $(\nu, \mu, \sigma,f)$ as defined before.  An automorphism $\psi=(\psi_1,\psi_2)$ of the relative Rota--Baxter group $\mathcal{A}$ defines a new $\mathcal{A}$-module structure on $\mathcal{K}$ given by $(\nu \psi_2, \mu \psi_1 ,\sigma \psi_2,f\psi_1)$, which we denote by $\mathcal{K}_{\psi}$, where 
\begin{eqnarray*}
\nu \psi_2(b) &=& \nu_{\psi_2(b)},\\
 \mu \psi_1(a) (k)   &=&  s_H(\psi_1(a))^{-1}  k  s_H(\psi_1(a)),\\
  \sigma \psi_2(b)(l)  &=& s_G(\psi_2(b))^{-1}ls_G(\psi_2(b)),\\
  f\psi_1(l,a)  &=& f(l,\psi_1(a)),
\end{eqnarray*}  
   for $a \in A$, $b\in B$, $k \in K$ and $l\in L$. It is not difficult to show that the automorphism $\psi$ induces an isomorphism $\psi^*$ of cohomology groups
$\psi^* : \Ho^2_{RRB}(\mathcal{A},\mathcal{K}) \rightarrow \Ho^2_{RRB}(\mathcal{A},\mathcal{K}_\psi)$
defined by
$$\phi^*([(\tau_1,\tau_2,\rho,\chi)])= [\big(\tau_1^{(\psi, \Id)}, \tau_2^{(\psi, \Id)},\rho^{(\psi, \Id)},\chi^{(\psi, \Id)}\big)].$$
Further, any $\mathcal{A}$-module isomorphism  $\theta=(\theta_1,\theta_2) : \mathcal{K} \rightarrow \mathcal{K}_\psi$ induces an isomorphism $\theta^*$ of cohomology groups
$\theta^* : \Ho^2_{RRB}(\mathcal{A},\mathcal{K}) \rightarrow \Ho^2_{RRB}(\mathcal{A},\mathcal{K}_\psi)$
given by
$$\theta^*([(\tau_1,\tau_2,\rho,\chi)]) = [\big(\tau_1^{(\Id, \theta^{-1})}, \tau_2^{(\Id, \theta^{-1})},\rho^{(\Id, \theta^{-1})},\chi^{(\Id, \theta^{-1})}\big)],$$
where $\theta^{-1}=(\theta_1^{-1},\theta_2^{-1})$. With this framework, we have the following result.

\begin{thm}\label{moduletheory}
	A pair of automorphisms $(\psi, \theta) \in \Aut(\mathcal{A}) \times \Aut(\mathcal{K})$  is inducible if and only if the following conditions hold:
\begin{enumerate}
\item   $\theta=(\theta_1,\theta_2) : \mathcal{K} \rightarrow \mathcal{K}_\psi $ is an isomorphism of $\mathcal{A}$-modules.
\item $\theta^*([(\tau_1,\tau_2,\rho,\chi)])= \psi^*([(\tau_1,\tau_2,\rho,\chi)])$.
\end{enumerate}
\end{thm}
\begin{proof}
	Let  $(\psi, \theta)$ be an inducible pair. Then there exists an automorphism $\gamma \in \Aut_{\mathcal{K}}(\mathcal{H})$ such that $(\psi, \theta) = (\gamma_{\mathcal{A}}, \gamma_{\mathcal{K}})$.   For $a \in A $, $b\in B$, $k \in K$ and $l\in L$, we have
	\begin{eqnarray}\label{modeqn1RRB}
		\gamma_1|_K(\nu_b(k))& =&\gamma_1|_K(\phi_{s_G(b)}(k))\nonumber\\
		&=&\phi_{\gamma_2(s_G(b))}(\gamma_1(k))\text{, since }(\gamma_1,\gamma_2)\in\Aut(\mathcal{H}) \nonumber\\
		&=&\phi_{s_G(\pi_2(\gamma_2(s_G(b))))}(\gamma_1(k)) \nonumber\\
		&=& \nu_{\gamma_{2_B}(b)}(\gamma_1|_K(k)), 	\text{ using that }\phi_{s_G(\pi_2(g))}=\phi_g\text{ for all }g\in G.			
	\end{eqnarray}
	Similarly, we have,
	\begin{eqnarray}
		\gamma_2|_L(\sigma_b(l))& = & \sigma_{\gamma_{2_B}(b)}(\gamma_2|_L(l)), \label{modeqn1RRBa }\\
			\gamma_1|_K(\mu_a(k))& =& \mu_{\gamma_{1_A}(a)}(\gamma_1|_K(k)), \label{modeqn1RRBb}\\
				\gamma_1|_K(f(l,a)) & = & f(\gamma_2|_L(l),\gamma_{1_A}(a))\label{modeqn1RRBc}.
	\end{eqnarray}

	This shows that  $((\Id,\Id), \gamma_{\mathcal{K}}) $ is compatible with the pairs of actions $(\nu, \mu, \sigma,f)$ and $(\nu \gamma_{2_B}, \mu \gamma_{1_A},\sigma \gamma_{2_B},f\gamma_{1_A})$, and hence condition (1) holds.
\par
	For each $a \in A$, there exists a unique element (say) $\kappa_1(a)$ in $K$ such that  $ \gamma_1(s_H(a)) =s_H(\gamma_{1_A}(a))  \kappa_1(a)$. Similarly, for each $b \in B$, there exists a unique element (say) $\kappa_2(b)$ in $L$ such that  $ \gamma_2(s_G(b)) =s_G(\gamma_{2_B}(b))  \kappa_2(b)$. It is clear that, for the maps $\kappa_1 : A \rightarrow K$ given by $a \mapsto \kappa_1(a)$ and $\kappa_2 : B \rightarrow L$ given by $b \mapsto \kappa_2(b)$, the pair $(\kappa_1,\kappa_2)$ lies in $C_{RRB}^1$. Given elements $h_1,h_2 \in H$, they have unique expressions of the form $h_1= s_H(a_1)  k_1$ and $h_2 = s_H(a_2)  k_2$ for some $a_1, a_2 \in A$ and $k_1, k_2 \in K$. Then, we have
	\begin{eqnarray*}
		\gamma_1(h_1h_2) &=& \gamma_1\big(s_H(a_1a_2)\mu_{a_2}(k_1)k_2\tau_1(a_1,a_2)\big)\\
		&=&\gamma_1(s_H(a_1a_2)) \gamma_1|_K\big(\mu_{a_2}(k_1)k_2\tau_1(a_1,a_2)\big)\\ 
		&=&s_H(\gamma_{1_A}(a_1)\gamma_{1_A}(a_2))\kappa_1(a_1a_2) \gamma_1|_K(\mu_{a_2}(k_1)k_2)\gamma_1|_K(\tau_1(a_1,a_2))
	\end{eqnarray*}
and
	\begin{eqnarray} 
		\gamma_1(h_1)\gamma_1(h_2)&=&s_H(\gamma_{1_A}(a_1))\kappa_1(a_1)\gamma_1|_K(k_1)s_H(\gamma_{1_A}(a_2))\kappa_1(a_2)\gamma_1|_K(k_2) \nonumber\\
		&=& s_H(\gamma_{1_A}(a_1))s_H(\gamma_{1_A}(a_2))\mu_{\gamma_{1_A}(a_2)}(\kappa_1(a_1)) \mu_{\gamma_{1_A}(a_2)}(\gamma_1|_K(k_1)) \kappa_1(a_2)  \gamma_1|_K(k_2) \nonumber\\
		&=& s_H(\gamma_{1_A}(a_1)\gamma_{1_A}(a_2))\tau_1(\gamma_{1_A}(a_1), \gamma_{1_A}(a_2))\mu_{\gamma_{1_A}(a_2)}(\kappa_1(a_1)) \nonumber\\
		&&  \; \mu_{\gamma_{1_A}(a_2)}(\gamma_1|_K(k_1)) \kappa(a_2)  \gamma_1|_K(k_2).\nonumber
	\end{eqnarray}
	Since $\gamma_1(h_1h_2) = \gamma_1(h_1)\gamma_1(h_2)$, the preceding two equations and \eqref{modeqn1RRBb}, give
	\begin{equation}\label{coholg1}
		\gamma_1|_K(\tau_1(a_1, a_2))(\tau_1(\gamma_{1_A}(a_1), \gamma_{1_A}(a_2)))^{-1}=\mu_{\gamma_{1_A}(a_2)}(\kappa_1(a_1))(\kappa_1(a_1a_2))^{-1}\kappa_1(a_2).
	\end{equation}
Similarly,	given elements $g_1,g_2 \in G$, they have unique expressions of the form $g_1= s_G(b_1)  l_1$ and $g_2 = s_G(b_2)  l_2$ for some $b_1, b_2 \in B$ and $l_1, l_2 \in L$.  Now, by similar computations, we get
	\begin{equation}\label{coholg2}
		\gamma_2|_L(\tau_2(b_1, b_2))(\tau_2(\gamma_{2_B}(b_1), \gamma_{2_B}(b_2)))^{-1}=\sigma_{\gamma_{2_B}(b_2)}(\kappa_2(b_1))(\kappa_2(b_1b_2))^{-1}\kappa_2(b_2).
	\end{equation}
Given an element $h\in H$, it can be uniquely written as $h=s_H(a)k$ for some $k\in K$. Then, we have
	\begin{eqnarray*}
		\gamma_2(R(h))&=&\gamma_2(R(s_H(a)k))\\
		&=&\gamma_2\big(s_G(T(a))\chi(a)S(\nu^{-1}_{T(a)}(k))\big)\mbox{, using Lemma \eqref{f rho chi eqn}}\\
		&=&\gamma_2(s_G(T(a)))\gamma_2|_L(\chi(a))\gamma_2|_L(S(\nu^{-1}_{T(a)}(k)))\\
		&=&s_G(\gamma_{2_B}(T(a)))\kappa_2(T(a))\gamma_2|_L(\chi(a))\gamma_2|_L(S(\nu^{-1}_{T(a)}(k))
	\end{eqnarray*}
	and
	\begin{eqnarray*}
		R(\gamma_1(h))&=&R(\gamma_1(s_H(a)k))\\
		&=&R(\gamma_1(s_H(a))\gamma_1(k))\\
		&=&R(s_H(\gamma_{1_A}(a))\kappa_1(a)\gamma_1|_K(k))\\
		&=&s_G(T(\gamma_{1_A}(a)))\chi(\gamma_{1_A}(a))S(\nu^{-1}_{T(\gamma_{1_A}(a))}(\kappa_1(a)\gamma_1|_K(k))).\\
	\end{eqnarray*}
	Since $\gamma_2(R(h))=R(\gamma_1(h))$, we have
	\begin{equation}\label{coholg3}
		\gamma_2|_L(\chi(a))(\chi(\gamma_{1_A}(a)))^{-1}=S\big(\nu^{-1}_{T(\gamma_{1_A}(a))}(\kappa_1(a))\big)\big(\kappa_2(T(a))\big)^{-1},
	\end{equation}
	as
\begin{eqnarray*}
	S(\nu^{-1}_{T(\gamma_{1_A}(a))}(\gamma_1|_K(k)))&=&	S(\nu^{-1}_{\gamma_{2_B}(T(a))}(\gamma_1|_K(k)))\\
	&=&S(\gamma_1|_K(\nu^{-1}_{T(a)}(k)))\\
	&=&\gamma_2|_L(S(\nu^{-1}_{T(a)}(k))).
\end{eqnarray*}
	Given $h\in H$ and $g\in G$, they can be uniquely written as $h=s_H(a)k$ and $g=s_G(b)l$ for some $a\in A, b\in B, k\in K$ and $l\in L$. Then, we have
	\begin{eqnarray*}
		\gamma_1\phi_g(h)&=&\gamma_1\big(s_H(\beta_b(a))\rho(a,b)\nu_b(f(l,a)k)\big)\mbox{, using Lemma \eqref{f rho chi eqn}}\\
		&=&\gamma_1\big(s_H(\beta_b(a))\big)\gamma_1|_K(\rho(a,b))\gamma_1|_K\big(\nu_b(f(l,a)k)\big)\\
		&=&s_H\big(\gamma_{1_A}(\beta_b(a))\big)\kappa_1(\beta_b(a))\gamma_1|_K(\rho(a,b))\gamma_1|_K\big(\nu_b(f(l,a)k)\big)
	\end{eqnarray*}
and
	\begin{eqnarray*}
	\phi_{\gamma_2(g)}\gamma_1(h)&=&\phi_{\gamma_2(s_G(b)l)}(\gamma_1(s_H(a)k))\\
	&=&\phi_{\gamma_2(s_G(b))\gamma_2|_L(l)}(\gamma_1(s_H(a))\gamma_1|_K(k))\\
	&=&\phi_{s_G(\gamma_{2_B}(b))\kappa_2(b)\gamma_2|_L(l)}(s_H(\gamma_{1_A}(a))\kappa_1(a)\gamma_1|_K(k))\\
	&=&s_H\big(\beta_{\gamma_{2_B}(b)}(\gamma_{1_A}(a))\big)\rho(\gamma_{1_A}(a),\gamma_{2_B}(b))\nu_{\gamma_{2_B}(b)}\big(f(\kappa_2(b)\gamma_2|_L(l),\gamma_{1_A}(a))\kappa_1(a)\gamma_1|_K(k)\big),\\
	&&\mbox{using Lemma \eqref{f rho chi eqn}.}
	\end{eqnarray*}
	Notice that,
	\begin{equation}\label{weq1}
	s_H\big(\gamma_{1_A}(\beta_b(a))\big)=s_H\big(\beta_{\gamma_{2_B}(b)}(\gamma_{1_A}(a))\big),
	\end{equation}
	since $(\gamma_{1_A},\gamma_{2_B})$ is an automorphism of $\mathcal{A}$. Using \eqref{modeqn1RRBc}, \eqref{weq1} and Lemma \eqref{properties of f}, we get
	\begin{equation}\label{coholg4}
		\gamma_1(\rho(a,b))\big(\rho(\gamma_{1_A}(a),\gamma_{2_B}(b))\big)^{-1}=\nu_{\gamma_{2_B}(b)}\big(f(\kappa_2(b),\gamma_{1_A}(a))\kappa_1(a)\big)\big(\kappa_1(\beta_b(a))\big)^{-1}.
	\end{equation}
Thus, \eqref{coholg1}, \eqref{coholg2}, \eqref{coholg3} and \eqref{coholg4} prove condition (2), that is, $\theta^*([(\tau_1,\tau_2,\rho,\chi)])= \psi^*([(\tau_1,\tau_2,\rho,\chi)])$.
\par

Conversely, let $(\psi,\theta) \in \Aut(\mathcal{A}) \times \Aut(\mathcal{K})$  satisfy conditions (1) and (2). Condition (2) guarantees the existence of maps $\kappa_1 : A \rightarrow K$ and $\kappa_2 : B \rightarrow L$ such that $(\kappa_1,\kappa_2)\in C^1_{RRB} $ and  
$$(\tau_1^{(\psi,\Id)}, \tau_2^{(\psi,\Id)}, \rho^{(\psi,\Id)}, \chi^{(\psi,\Id)} ) \,(\tau^{(\Id,\theta^{-1})}_1, \tau^{(\Id,\theta^{-1})}_2, \rho^{(\Id,\theta^{-1})}, \chi^{(\Id,\theta^{-1})} )^{-1}=\partial^1_{RRB}(\kappa_1, \kappa_2).$$
Define the maps $\gamma_1:H\rightarrow H$ by
$$\gamma_1(s_H(a)k)=s_H(\psi_1(a))\kappa_1(a)\theta_1(k)$$
$\text{and }\gamma_2:G\rightarrow G$ by
$$\gamma_2(s_G(b)l)=s_G(\psi_2(b))\kappa_2(b)\theta_2(l)$$
for all $a\in A$, $b \in B$, $k\in K$ and $l \in L$. It is easy to see that $\gamma=(\gamma_1,\gamma_2)\in \Aut_{\mathcal{K}}(\mathcal{H})$ and  $\gamma_{\mathcal{A}}=\psi,$  $\gamma_{\mathcal{K}}=\theta$. This completes the proof of the theorem.
\end{proof}

\begin{ack}
	{\rm Pragya Belwal thanks UGC for the PhD research fellowship and Nishant Rathee thanks IISER Mohali for the institute post doctoral fellowship. The authors express their gratitude to Professor Mahender Singh for his valuable and  insightful comments. }
\end{ack}
\medskip
\section{Declaration}
The authors declare that they have no conflicts of interest.

\end{document}